\documentclass[11pt]{amsart}
\usepackage{amsmath, amssymb, amsthm, graphicx, color, tikz, pb-diagram}
\usepackage{fullpage}
\usepackage[normalem]{ulem}
\usepackage{float, placeins}
\usepackage{commath}
\usepackage{cite}
\usepackage{comment}
\usepackage{array}
\usepackage{caption}
\usepackage{booktabs}
\usepackage{longtable}
\usepackage[hyperfootnotes=false, colorlinks, linkcolor={blue}, citecolor={magenta}, filecolor={blue}, urlcolor={blue}, plainpages=false, pdfpagelabels]{hyperref}
\usepackage{makecell}

\allowdisplaybreaks

\newtheorem{theorem}{Theorem}[section]
\newtheorem{lemma}[theorem]{Lemma}

\newtheorem{corollary}[theorem]{Corollary}
\theoremstyle{definition}

\newtheorem{remark}[theorem]{Remark}

\newtheorem{example}[theorem]{Example}

\newtheorem*{theorem*}{Theorem}
\newtheorem*{remark*}{Remark}
\newtheorem*{remarks*}{Remarks}
\newtheorem*{definition*}{Definition}

\newcommand*\HYPERskip{&}
\catcode`,\active
\newcommand*\pFq{
\begingroup
\catcode`\,\active
\def ,{\HYPERskip}%
\doHyper
}
\catcode`\,12
\def\doHyper#1#2#3#4#5{%
\, _{#1}F_{#2}\left[\begin{matrix}#3 \smallskip \\  #4\end{matrix} \; ; \; #5\right]%
\endgroup
}

\newcommand{\ZZ}{\mathbb{Z}}     
\newcommand{\RR}{\mathbb{R}}     
\newcommand{\NN}{\mathbb{N}}     
\newcommand{\PP}{\mathbb{P}}      
\newcommand{\QQ}{\mathbb{Q}}      
\newcommand{\CC}{\mathbb{C}}      
\newcommand{\hh}{\mathfrak{h}}  

\def\O{\mathcal O} 
\def\G{\Gamma} 
\def\l{\lambda}

\def\mat#1#2#3#4{\begin{pmatrix} #1&#2\\#3&#4\end{pmatrix}}
\def\smat#1#2#3#4{\left(\begin{smallmatrix} #1&#2\\#3&#4\end{smallmatrix}\right)}
\def\({\left(}
\def\){\right)}

\DeclareMathOperator{\SL}{SL}
\DeclareMathOperator{\GL}{GL}
\newcommand{\PSL}{\mathrm{PSL}}

\makeatletter
\@namedef{subjclassname@2020}{%
  \textup{2020} Mathematics Subject Classification}
\makeatother

\title[]{Generalized Ramanujan-Sato Series Arising from Modular Forms }

\author{Angelica Babei}
\address{Department of Mathematics and Statistics, McMaster University, Hamilton, Ontario L8S 4L8, Canada}
\email{babeia@mcmaster.ca}

\author{Lea Beneish} 
\address{Department of Mathematics, University of California, Berkeley, Berkeley, CA 94720, USA}
\email{leabeneish@math.berkeley.edu}

\author{Manami Roy}
\address{Department of Mathematics, Fordham University, Bronx, New York 10458, USA}
\email{mroy17@fordham.edu}

\author{Holly Swisher}
\address{Department of Mathematics, Oregon State University, Corvallis, OR, 97301, USA}
\email{swisherh@oregonstate.edu}

\author{Bella Tobin}
\address{Department of Mathematics, Oregon State University, Corvallis, OR, 97301, USA}
\email{tobinbe@oregonstate.edu}

\author{Fang-Ting Tu}
\address{Department of Mathematics, Louisiana State University, Baton Rouge, LA 70803, USA}
\email{ftu@lsu.edu}

\thanks{The authors thank the Women in Numbers 5 Workshop (WIN5) for the opportunity to initiate this collaboration. The work of the first author was supported by a Simons Collaboration Grant (550029, to John Voight). The fourth author is appreciative of support by the National Science Foundation Grant DMS-2101906.}

\begin{document}
\begin{abstract}
Motivated by work of Chan, Chan, and Liu, we obtain a new general theorem which produces Ramanujan-Sato series for $1/\pi$.  We then use it to construct explicit examples related to non-compact arithmetic triangle groups, as classified by Takeuchi. Some of our examples are new, and some reproduce existing examples.
\end{abstract}

\date{\today}

\maketitle

\section{Introduction and Statement of Results}

Ramanujan \cite{Ramanujan} gave a list of infinite series identities of the form
\[
\sum_{k= 0}^{\infty} \frac{(\frac{1}{2})_k (\frac{1}{d})_k (\frac{d-1}{d})_k }{k!^3} (ak+1)(\lambda_d)^k=\frac{\delta}{\pi},
\]
for $d=2,3,4,6,$ where $\l_d$ are singular values that  correspond to elliptic curves with complex multiplication,
$a, \delta$ are explicit algebraic numbers, and $(a)_k$ denotes the rising factorial $(a)_k=a(a+1)\cdots(a+k-1)$.  In fact, a similar identity was given even earlier by Bauer \cite{Bauer}.  Proofs of these formulas were first given by J. Borwein and P. Borwein \cite{BB} and D. Chudnovsky and G. Chudnovsky \cite{CC}, and both approaches rely on the arithmetic of elliptic integrals of the first and second kind, including the Legendre relation at singular values.  Since the 1980's series of this type have been at the forefront of algorithms to compute decimal approximations of $\pi$.  Although Ramanujan indicated that there were general theories behind such series, this remains not fully understood.  Deriving new series for $1/\pi^k$ and the unifying theories underlying such series is an active research area (see \cite{Zudilin}, \cite{ChanCooper} for example).  

In work of Chan, Chan, and Liu \cite{CCL}, motivated by Sato, the authors derive a general series for $1/ \pi$ that admits existing series as special cases.  They give a table of explicitly computed examples, each relating a hauptmodul and hypergeometric function connected to an index $2$ subgroup pair of triangle groups.

Here, we use the method of Chan, Chan, and Liu \cite{CCL} to obtain a new general theorem which produces series formulas for $1/ \pi$ and as an application construct explicit Ramanujan-Sato type formulas for $1/\pi$ related to other non-compact arithmetic triangle groups as classified by Takeuchi \cite{Takeuchi, Takeuchi77}. Some of our examples are new, and some reproduce existing examples.

Throughout the paper we denote the upper half plane by $\hh=\{ \tau = x+iy \in \CC: y > 0\}$.  Our main theorem is the following.

\begin{theorem}\label{thm: general}  
Let $Z(\tau)$, $X(\tau)$, $U(\tau)$ be meromorphic on $\hh$ such that $\frac{1}{2\pi i}\frac{d}{d\tau}X(\tau) = U(\tau)X(\tau)Z(\tau)$ and when $\tau\in D$, a domain of $\hh$, $Z(\tau)= X(\tau)^{\varepsilon_0}(1-X(\tau))^{\varepsilon_1}\sum\limits_{j=0}^\infty A_jX(\tau)^j$ for $\varepsilon_0,\varepsilon_1\in \RR$, $A_j\in \CC$. Further assume there exist $\alpha,\beta\in \CC$, $N,k \in \NN$, and elements  $\gamma = \left(\begin{smallmatrix}a&b\\c&-a\end{smallmatrix}\right), \delta = \left(\begin{smallmatrix}1 & \frac ac(1-N) \\ 0 & 1 \end{smallmatrix}\right) \in \SL_2(\RR)$ such that for any $\tau \in \hh$,
\begin{align*}
\left( Z \vert_k  \gamma \right) (\tau) &= \alpha Z (\tau), \\
\left( Z \vert_k  \delta \right) (\tau) & =  \beta Z\left(\tau\right).
\end{align*}
Set $M_N(\tau):=Z(\tau)/Z(N\tau)$, and let $\tau_0=\frac{a}{c}+ \frac{i}{c\sqrt N}$.  Then it follows that if $\tau_0\in D$,
$$
\frac{ck\sqrt N}{2\pi} =U(\tau_0)X(\tau_0)^{\varepsilon_0}(1-X(\tau_0))^{\varepsilon_1}\sum_{j\geq 0} (2j+a_N) A_j X(\tau_0)^j,
$$
where
\[
a_N = 2\left(\varepsilon_0+\varepsilon_1\frac{X(\tau_0)}{X(\tau_0)-1}\right) -{\frac{\alpha i^k}{\beta N^{k/2}}} X(\tau_0) \left(\frac{dM_N}{dX}\right)\left|_{X=X(\tau_0)}. \right.
\]
Alternatively if $\gamma\tau_0\in D$,
$$
\frac{ck\sqrt N}{2\pi}=X(\gamma\tau_0)^{\varepsilon_0}(1-X(\gamma\tau_0))^{\varepsilon_1}\sum_{j\geq 0} (b_N'j+a'_N) A_j X(\gamma\tau_0)^j,
$$
where
\begin{align*}
b_N' & = 2NU(\gamma\tau_0), \\
a'_N &= 2N U(\gamma\tau_0)\left(\varepsilon_0+\varepsilon_1\frac{X(\gamma\tau_0)}{X(\gamma\tau_0)-1}\right) + \frac{1}{\beta} U(\tau_0)X(\tau_0) \left(\frac{dM_N}{dX}\right) \left|_{X=X(\tau_0)}. \right.
\end{align*}
\end{theorem}

\medskip

The rest of the paper is organized as follows.  In Section \ref{sec:prelim} we provide some relevant background material and give a few lemmas which will be useful for proving Theorem \ref{thm: general} and for our applications.  In Section \ref{sec:pf1}, we prove Theorem \ref{thm: general}. In Sections \ref{sec:exs2} and \ref{sec:exs1} we give several applications by constructing explicit examples of Ramanujan-Sato series for $1/\pi$ related to certain arithmetic triangle groups.  The groups we use to construct examples are among the arithmetic triangle groups commensurable with $\PSL_2(\ZZ)\cong(2,3,\infty)$ coming from Takeuchi's class I (those of non-compact type). These have the following subgroup diagram. 
$$
  
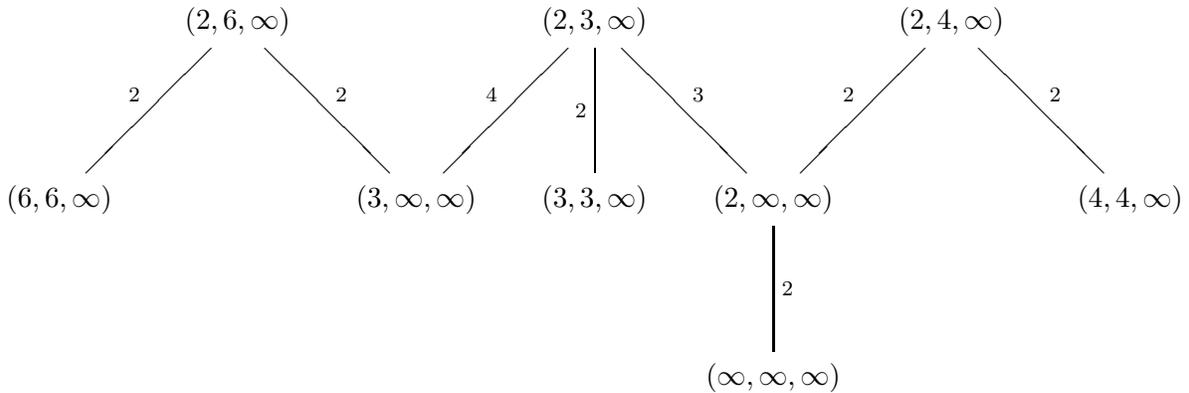
\begin{diagram}
  \node[2]{(2,6,\infty)}  \arrow{sw,l,-}{2}  \arrow{se,l,-}{2} 
  \node[2]{(2,3,\infty)}   \arrow{sw,l,-}{4}  \arrow{s,l,-}{2} \arrow{se,l,-}{3} 
   \node[2]{(2,4,\infty)}   \arrow{sw,l,-}{2}  \arrow{se,l,-}{2} \\ 
   \node[1]  {(6,6,\infty)} 
   \node[2]  {(3,\infty,\infty)} 
    \node[1]  {(3,3,\infty)} 
   \node[1]   {(2,\infty,\infty)} \arrow{s,r,-}2 
   \node[2]  {(4,4,\infty)} \\ 
   \node[5]{(\infty,\infty,\infty)}
  \end{diagram}
$$
In particular, in Section \ref{sec:exs2} we construct examples of series for $1/\pi$ arising from modular forms for the groups $\G_0(2) \cong (2,\infty,\infty)$, $\G_0(3) \cong (3,\infty,\infty)$, and $\G_0(4) \cong (\infty,\infty,\infty)$. In Section \ref{sec:exs1} we construct examples arising from modular forms for the groups $\rm{PSL}_2(\mathbb{Z}) \cong (2,3,\infty)$, $\G_0(2)^+ \cong (2,4,\infty)$, and $\G_0(3)^+ \cong (2,6,\infty)$.

\section{Preliminaries} \label{sec:prelim}

Suppose $\G$ is a discrete subgroup of $\rm{SL}_2(\RR)$ having genus zero that is commensurable with a subgroup $\Gamma(\O)$ of $\rm{SL}_2(\RR)$ arising from a norm $1$ group of a quaternion order $\O$. The group $\Gamma$ acts on the upper half plane $\hh$ and $\PP^1(\RR)$ by linear fractional transformations $\gamma\cdot \tau=\frac{a\tau+b}{c\tau+d}$, where $\gamma=\smat abcd\in \G $.  A classical result from the theory of compact Riemann surfaces says that when $\Gamma$ has genus zero there exist finitely many elliptic or parabolic elements $r_1,\ldots,r_k$ that generate $\Gamma/\{\pm 1\}$ with the relations
$r_1\ldots r_k=1$, $r_i^{e_i}=1$. We call $(0;e_1,\ldots,e_k)$ the \emph{signature} of $\Gamma/\{\pm 1\}$.  Work of Yang \cite{Yang-Schwarzian} has shown that the modular forms on $\Gamma$ can be expressed in terms of a hauptmodul $t$ of $\G$ and solutions of the Schwarzian differential equation 
$$2Q(t)t'(\tau)^2+\{t,\tau\}=0,$$ where
\[
\{t,\tau\}=\frac{t'''(\tau)}{t'(\tau)}-\frac32\left(\frac{t''(\tau)}{t'(\tau)}\right)^2
\]
is the Schwarzian derivative.    If $\Gamma$ has signature $(0;e_1,e_2,e_3)$, then this differential equation is hypergeometric and one can describe the modular forms on $\Gamma$ using hypergeometric functions. 

Takeuchi \cite{Takeuchi, Takeuchi77} has classified arithmetic triangle groups, including their quaternion orders and inclusion relations between them.  

\subsection{Arithmetic triangle groups and a theorem of Yang} 

Arithmetic triangle groups are certain discrete subgroups of $\PSL_2(\RR)$. Consider integers $e_1, e_2, e_3 >1$, possibly $\infty$, such that $\frac{1}{e_1}+\frac{1}{e_2}+\frac{1}{e_3} <1$. Then there exists a triangle $S$  in $\hh$ with internal angles $\frac{\pi}{e_1}, \frac{\pi}{e_2}$ and $\frac{\pi}{e_3}$, where $\frac{\pi}{\infty}:=0$. The group of symmetries of the tiling of $\hh$ by triangles congruent to $S$ is called a triangle group, and can be presented in terms of generators 
\[ (e_1, e_2, e_3)=\langle r_1, r_2, r_3 \, | \,  r_1^{e_1}=r_2^{e_2}=r_3^{e_3}=r_1r_2r_3=1\rangle,\] where if one of  $e_i=\infty$ for $i=1,2,3$, the relation $r_i^{e_i}=1$ is trivial. Takeuchi \cite{Takeuchi, Takeuchi77} found $85$ triples, up to permutation, where the corresponding group $(e_1, e_2, e_3) \le \PSL_2(\RR)$ is arithmetic. In this context, we define 
$$
\Gamma_0(N):=\left\{A=\mat abcd\in \SL_2(\ZZ):\, A\equiv \mat\ast\ast{0}\ast (\bmod N)  \right\}\slash \{\pm 1\}.
$$
For example, $(2, \infty, \infty) \cong \Gamma_0(2)$ and $(\infty, \infty, \infty) \cong \Gamma_0(4)$. We direct the reader to \cite{Takeuchi, Takeuchi77} for more background on arithmetic triangle groups. 


Throughout, we call an elliptic point of order $\infty$ a cusp. The following theorem of Yang allows us to write modular forms on $T$ in terms of $_2F_1$ hypergeometric functions evaluated at specific Hauptmoduln $X(\tau)$. 

\begin{theorem}[Yang{\cite[Thm. 9]{Yang-Schwarzian}}] \label{theorem: triangle}
  Assume that  $\G$ has signature
  $(0;e_1,e_2,e_3)$. Let $X(\tau)$ be the Hauptmodul of $\G$ with
  values $0$, $1$, and $\infty$ at the elliptic points of order $e_1$,
  $e_2$, and $e_3$ (possibly $\infty$), respectively. Let $k\ge 2$ be an even integer.
  Then a basis for the space of modular forms of weight $k$ on $\G$
  is given by
  $$
    X^{\{k(1-1/e_1)/2\}}(1-X)^{\{k(1-1/e_2)/2\}}X^j\left(
    \pFq21{a&b}{&c}X+CX^{1/e_1} \pFq21{a'&b'}{&c'}X\right)^k,
  $$
  $j=0,\ldots, \mathcal{D}-1$, 
  for some constant $C$, where
  for a rational number $x$, we let $\{x\}$ denote the fractional part
  of $x$,
  $$
    a=\frac12\left(1-\frac1{e_1}-\frac1{e_2}-\frac1{e_3}\right), \qquad
    b=a+\frac1{e_3}, \qquad c=1-\frac1{e_1},
  $$
  $$
    a'=a+\frac1{e_1}, \qquad b'=b+\frac1{e_1}, \qquad
    c'=c+\frac2{e_1},
  $$
and 
  $$
    \mathcal{D}=1-k+\lfloor k(1-1/e_1)/2\rfloor+\lfloor k(1-1/e_2)/2\rfloor+\lfloor k(1-1/e_3)/2\rfloor
  $$
 is the dimension of the space of weight-$k$ modular forms for $\G$. 
\end{theorem}

In Table \ref{tab:ell_pts} we catalog the relevant elliptic point data needed to use Theorem \ref{theorem: triangle} for the groups $\G$ we consider in Sections \ref{sec:exs2} and \ref{sec:exs1}. In Table \ref{tab:ell_pts}, the ``Stabilizing elements" are the generators of the isotropy subgroups of the group $\Gamma$ (the stabilizer in the group $\Gamma$) that fix the corresponding elliptic points under the group action of $\Gamma$ on $\hh \cup \PP^1(\RR)$. The ``Orders" are the orders of the isotropy subgroups. For example, for the group $\G_0(2)$, there is one in-equivalent elliptic point represented by $\frac{1+i}2$ and two in-equivalent cusps, represented by $0$ and $i\infty$. The elliptic point $\frac{1+i}2$  is fixed by all elements of the subgroup $\langle\smat 1{-1}2{-1}\rangle$, the cusp $0$ is fixed by all elements of $\langle\smat {-1}02{-1}\rangle$, and $i\infty$ is fixed by all elements of $\langle \smat 110{1} \rangle$.

\renewcommand{\arraystretch}{1.6}

\begin{table}[h!]
\caption{Elliptic point data}
\label{tab:ell_pts}
\begin{tabular}{|c|c|c|c|}
\hline
Group $\G$ & Elliptic points & Stabilizing elements & Orders \\
\hline
\hline
$\G_0(2) \cong (2,\infty,\infty)$ & $\frac{1+i}{2}$, $0$, $i\infty$ & $\left(\begin{smallmatrix} 1 & -1 \\ 2 & -1\end{smallmatrix}\right)$, $\left(\begin{smallmatrix} -1 & 0 \\ 2 & -1\end{smallmatrix}\right)$, $\left(\begin{smallmatrix} 1 & 1 \\ 0 & 1\end{smallmatrix}\right)$ & $2$, $\infty$, $\infty$ \\
\hline
$\G_0(3) \cong (3,\infty,\infty)$ & $\frac{3+i\sqrt 3}{6}$, $0$, $i\infty$ & $\left(\begin{smallmatrix} -1 & 1 \\ -3 & 2\end{smallmatrix}\right)$, $\left(\begin{smallmatrix} -1 & 0 \\ 3 & -1\end{smallmatrix}\right)$, $\left(\begin{smallmatrix} 1 & 1 \\ 0 & 1\end{smallmatrix}\right)$ & $3$, $\infty$, $\infty$ \\
\hline
$\G_0(4) \cong (\infty,\infty,\infty)$ & $0$, $\frac 12$, $i\infty$ & $\left(\begin{smallmatrix} 1 & 0 \\ 4 & 1\end{smallmatrix}\right)$, $\left(\begin{smallmatrix} 1 & -1 \\ 4 & -3\end{smallmatrix}\right)$, $\left(\begin{smallmatrix} 1 & 1 \\ 0 & 1\end{smallmatrix}\right)$ & $\infty$, $\infty$, $\infty$ \\
\hline
$\rm{PSL}_2(\mathbb{Z}) \cong (2, 3,\infty)$ & $i$, $\frac{-1+i\sqrt{3}}{2}$, $i\infty$ & $\left(\begin{smallmatrix} 0 & -1 \\ 1 & 0\end{smallmatrix}\right)$, $\left(\begin{smallmatrix} 0 & -1 \\ 1& 1\end{smallmatrix}\right)$, $\left(\begin{smallmatrix} 1 & 1 \\ 0 & 1\end{smallmatrix}\right)$ & $2$, $3$, $\infty$ \\
\hline
$\G_0(2)^+ \cong (2,4,\infty)$ & $\frac{i}{\sqrt 2}$, $\frac{1+i}{2}$,  $i\infty$ & $\frac{1}{\sqrt 2}\left(\begin{smallmatrix} 0 & -1 \\ 2& 0\end{smallmatrix}\right)$, $ \sqrt 2\left(\begin{smallmatrix} 0 & -\frac12 \\ 1 & -1\end{smallmatrix}\right)$, $\left(\begin{smallmatrix} 1 & 1 \\ 0 & 1\end{smallmatrix}\right)$  & $2$, $4$, $\infty$ \\
\hline
$\G_0(3)^+ \cong (2,6,\infty)$ & $\frac{i}{\sqrt 3}$, $\frac{3+i\sqrt 3}{6}$,  $i\infty$&  $\frac{1}{\sqrt 3}\left(\begin{smallmatrix} 0 & -1 \\ 3& 0\end{smallmatrix}\right)$, $ \sqrt 3\left(\begin{smallmatrix} 0 & -\frac13 \\ 1 & -1\end{smallmatrix}\right)$, $\left(\begin{smallmatrix} 1 & 1 \\ 0 & 1\end{smallmatrix}\right)$ & $2$, $6$, $\infty$ \\
\hline
\end{tabular}
\end{table}

Each of the groups $\G$ listed in Table \ref{tab:ell_pts} yield a modular curve of genus $0$, and thus the function field of this modular curve is generated by a single modular function for $\G$, a Hauptmodul.  By the theory of Riemann surfaces \cite[Prop. 1.21]{Miranda} any nonconstant modular function for $\G$ that has a single simple pole is a Hauptmodul for $\G$.  

In order to apply Theorem \ref{theorem: triangle} to our groups $\G$, we need to choose a Hauptmodul which takes the values $0$, $1$, and $\infty$ at the elliptic points of $\G$.  Then assign $e_1$, $e_2$, $e_3$ to be the orders of the elliptic points yielding the values $0$, $1$, $\infty$, respectively.  We list this choice of Hauptmodul $X(\tau)$ for each group $\G $, along with the values of $e_1$, $e_2$, $e_3$ in Table \ref{tab:Hauptmoduln}, where $\eta(\tau)$ is the Dedekind eta function $ \eta(\tau):= e^{\pi i \tau/12}\prod_{n =1}^\infty (1-e^{2\pi i n\tau})$.  That these are Hauptmoduln can be checked directly or observed in \cite{Maier} and \cite{ConwayNorton}. 

\renewcommand{\arraystretch}{1.6}

\begin{table}[h!]
\caption{Choice of Hauptmoduln}
\label{tab:Hauptmoduln}
\begin{tabular}{|r|l|c|}
\hline
Group $\G$ & Hauptmodul& $e_1$, $e_2$, $e_3$ \\
\hline
\hline
$\G_0(2) \cong (2,\infty,\infty)$ & $t_2(\tau) = -64\eta(2\tau)^{24}/\eta(\tau)^{24}$ & $\infty$, $2$, $\infty$\\
\hline
$\G_0(3) \cong (3,\infty,\infty)$ & $t_3(\tau) = -27\eta(3\tau)^{12}/\eta(\tau)^{12}$ & $\infty$, $3$, $\infty$ \\
\hline
$\G_0(4) \cong (\infty,\infty,\infty)$ & $t_\infty(\tau) = 16\eta(\tau)^8\eta(4\tau)^{16}/\eta(2\tau)^{24}$ & $\infty$, $\infty$, $\infty$ \\
\hline
$\rm{PSL}_2(\mathbb{Z}) \cong (2, 3,\infty)$ & $t_{2,3}(\tau) = 1728/j(\tau)$ & $\infty$, $2$, $3$  \\
\hline
$\G_0(2)^+ \cong (2,4,\infty)$ & $t_{2,4}(\tau) = 256\eta(\tau)^{24}\eta(2\tau)^{24}/(\eta(\tau)^{24}+64\eta(2\tau)^{24})^2$ & $\infty$, $2$, $4$\\
\hline
$\G_0(3)^+ \cong (2,6,\infty)$ & $t_{2,6}(\tau) = 108\eta(\tau)^{12}\eta(3\tau)^{12}/(\eta(\tau)^{12}+27\eta(3\tau)^{12})^2$ & $\infty$, $2$, $6$\\
\hline
\end{tabular}
\end{table}

We note the following relationship between the $j$-function and $t_2$, as defined in Table \ref{tab:Hauptmoduln}, which was observed by Maier \cite{Maier}\footnote{The $t_2$ in \cite{Maier} is equal to a constant multiple of $-2^6$ from our $t_2$}
\begin{equation}
    \label{eqn: jt2reln}
    j=\frac{64(4t_2-1)^3}{t_2}.
\end{equation}

We now use Theorem \ref{theorem: triangle} and the Hauptmoduln in Table \ref{tab:Hauptmoduln} to compute generators of spaces of modular forms for our groups $\G$.  For each group, we choose the least weight $k$ that yields nontrivial modular forms.  The generators obtained are listed in Table \ref{tab:gens}.  To save space, we use the notation 
\[
F[a,b;c;x]:= \pFq21{a & b}{&c}x.
\]

\begin{table}[h!]
\caption{Hypergeometric Modular Form Generators}
\label{tab:gens}
\begin{tabular}{|c|c|c|}
\hline
Group $\G$ & Weight $k$ & Generator(s) of weight $k$ modular forms for $\G$ \\
\hline
\hline
$\G_0(2) \cong (2,\infty,\infty)$ & $2$ & $(1-t_2)^\frac12 F[\frac14, \frac14; 1; t_2]^2$ \\
\hline
$\G_0(3) \cong (3,\infty,\infty)$ & $2$ & $(1-t_3)^\frac23 F[\frac13, \frac13; 1; t_3]^2$ \\
\hline
$\G_0(4) \cong (\infty,\infty,\infty)$ & $2$ & $F[\frac12, \frac12; 1; t_\infty]^2$, $t_\infty F[\frac12, \frac12; 1; t_\infty]^2$  \\
\hline
$\rm{PSL}_2(\mathbb{Z}) \cong (2, 3,\infty)$ & $4$ & $F[\frac{1}{12}, \frac{5}{12}; 1; t_{2,3}]^4$ \\
\hline
$\G_0(2)^+ \cong (2,4,\infty)$ & $4$ & $F[\frac18, \frac38;1;t_{2,4}]^4$ \\
\hline
$\G_0(3)^+ \cong (2,6,\infty)$ & $4$ & $F[\frac16, \frac13; 1; t_{2,6}]^4$ \\
\hline
\end{tabular}
\end{table}

\subsection{Connection to known modular forms and their properties}

In our examples, we will use some known modular forms. For even $k \ge 4$, the Eisenstein series of weight $k$  is a modular form (of weight $k$) for $\PSL_2(\ZZ)$ defined \cite{DiamondShurman} as
 \begin{equation}
 \label{eqn: eseries}
    E_k(\tau) := \frac{1}{2\zeta(k)} \sum_{(m, n) \in \ZZ^2-(0,0)} \frac{1}{(m\tau+n)^k}=1-\frac{2k}{B_k}\sum_{n\ge 1} \sigma_{k-1}(n)e^{2\pi i n \tau},  
 \end{equation} where $\zeta(s)$ is the Riemann zeta function,  $B_k$ is the $k$th Bernoulli number, and $\sigma_{k-1}(n)$ is the arithmetic function
 \[\sigma_{k-1}(n)=\sum_{d|n} d^{k-1}.\] The latter equality in \eqref{eqn: eseries} follows since $E_k(\tau)$ converges absolutely for $k \ge 4$.

When $k=2$, we define $$E_2(\tau)=1-24\sum_{n\ge 1} \sigma_1(n)e^{2\pi i n \tau}.$$ For any positive number $N$, 
\begin{equation}
\label{eqn: E2N}
E_{2, N}:= E_2(\tau)-NE_2(N\tau)
\end{equation} is a modular form of weight $2$ for $\Gamma_0(N)$ \cite[Exercise 1.2.8]{DiamondShurman}.

The Jacobi $\theta$-functions are defined as 
	$$
	\theta_2(\tau):=\sum_{n\in\mathbb Z} e^{\pi i (n+1/2)^2 \tau},\quad
	\theta_3(\tau):=\sum_{n\in\mathbb Z} e^{\pi i n^2\tau}, \quad
	\theta_4(\tau):=\sum_{n\in\mathbb Z} (-1)^ne^{\pi i n^2\tau}.
	$$ They can be expressed as eta quotients via \cite[p.29]{BGHZ} 
	\begin{equation}\label{eqn: thetaeta}
	    \theta_2(\tau)=2\frac{\eta(2\tau)^2}{\eta(\tau)}, \quad \theta_3(\tau)=\frac{\eta(\tau)^5}{\eta(\tau/2)^2\eta(2\tau)^2}, \quad \theta_4(\tau)=\frac{\eta(\tau/2)^2}{\eta(\tau)},
	\end{equation}
	and satisfy the Jacobi identity \cite[p.28]{BGHZ}
	
		\begin{equation}\label{eqn: jacobiidentity}
	    \theta_2^4+\theta_4^4=\theta_3^4.
	\end{equation}

Moreover, the modular $\l$-function is given by
\begin{equation}
    \label{defn: lambda}
    \lambda=\left(\frac{\theta_2}{\theta_3}\right)^4,\quad 	1-\lambda =\left(\frac{\theta_4}{\theta_3}\right)^4,
\end{equation}
	and satisfies the transformation formula  \cite[p.109]{Chandrasekharan}
\begin{equation}
    \label{eqn: lambdatransform}
    \lambda\left( \frac{-1}{\tau}\right)=1-\lambda(\tau).
\end{equation}
	
The function $\lambda(\tau)$ is a Hauptmodul for the group $\G(2)$, which has cusps 0, 1,  $i\infty$ at which the values of $\l$ are $1$, $\infty$, $0$ respectively \cite[Chapter VII, \S 7-8]{Chandrasekharan}.  Moreover, the function $\lambda(2\tau)$ is a Hauptmodul for $\Gamma_0(4)$, and one can check $\lambda(2\tau)$ equals the choice of Hauptmodul $t_\infty$ for $\Gamma_0(4)$ given in Table \ref{tab:Hauptmoduln}. Furthermore, each $\theta_i(2\tau)$ is a modular form of weight $1/2$ for $\Gamma_0(4)$ \cite[p.28]{BGHZ}.

We also note the following classical results,
	
\begin{equation}
\label{eqn: thetaE2}
	\theta_3^4(2\tau)=\frac{4E_2(4\tau)-E_2(\tau)}{3},\quad \theta_3^4(2\tau)+\theta_2^4(2\tau)=\frac 12 \left(\theta_3^4(\tau)+ \theta_4^4(\tau)\right)=-E_{2,2}(\tau),
\end{equation}
	
\begin{equation}
\label{eqn: thetaE23}
	(\theta_3(2\tau)\theta_3(6\tau)+\theta_2(2\tau)\theta_2(6\tau))^2=\frac{1}{2} E_{2,3}(\tau) = -\frac{(3\eta(3\tau)^3+\eta(\tau/3)^3)^2}{\eta^2(\tau)},
\end{equation}

\begin{equation}
\label{eqn: thetahypgeom}  
\pFq{2}{1}{\frac 1{2} &\frac 12}{&1}{\lambda(2\tau)}^2=\theta_3(2\tau)^4, \quad  (1- \l(\tau)) \pFq{2}{1}{\frac 1{2} &\frac 12}{&1}{\lambda(2\tau)}^2=\theta_4(2\tau)^4.
\end{equation}

Identities such as those above can be checked using Sturm-type bound arguments, which we employ throughout the article. A theorem of Choi and Kim \cite{ChoiKim} gives a Sturm-type bound for computationally determining when modular forms on genus $0$ groups $\G_0(N)^+$ are equal. Since their bound is less than Sturm's bound \cite{Sturm} for $\G_0(N)$, we give special cases of these results as follows. 

Given a formal sum $f=\sum_{n\gg -\infty} c(n)q^n$, define
\[
\text{ord}_{i\infty} f:= \inf\{n\in \ZZ \mid c(n)\neq 0 \}.
\]

\begin{theorem}[Sturm \cite{Sturm} and Choi, Kim \cite{ChoiKim}] \label{thm:Sturm}

Let $N\in \NN$ such that $\G_0(N)^+$ has genus $0$. Suppose $f$ is a meromorphic modular form of weight $k$ for $\G\in \{\G_0(N),\G_0(N)^+\}$ having poles only at the cusp $i\infty$. If $f$ has integer Fourier coefficients and
\[
\begin{cases}
\text{ord}_{i\infty} f > \frac{k}{12}[\SL_2(\ZZ):\G_0(N)] & \text{when } \G=\G_0(N), \\
\text{ord}_{i\infty} f > \frac{k}{24}[\SL_2(\ZZ):\G_0(N)] & \text{when } \G=\G_0(N)^+,
\end{cases}
\]
then $f=0$.
\end{theorem}

\subsection{Some useful results} 

The following formula due to Clausen \cite{AAR} states that 
\begin{equation}\label{eq:Clausen}
\pFq21{a&b}{&a+b+\frac12}{z}^2=\pFq32{2a&2b&a+b}{&2a+2b&a+b+\frac 12}{z}.
\end{equation}

We next give two useful lemmas. The first we use in the proof of Theorem \ref{thm: general}.  

\begin{lemma}\label{lem:slashderiv}
Suppose $f(\tau)$ is a meromorphic function on $\hh$.  If there exists $A=\left(\begin{smallmatrix} a&b\\c&d\end{smallmatrix}\right)\in \SL_2(\RR)$, $k\in \NN$, and $\beta \in \CC$ such that for all $\tau\in \hh$,
\[
(f|_kA)(\tau)= \beta f(\tau),
\]
then for any $\tau\in \hh$ that is not a pole of $f$,
\[
\frac 1{\beta(c\tau+d)^2}\frac{df}{d\tau}(A\tau)=(c\tau+d)^k \frac{df}{d\tau}(\tau)+kc(c\tau+d)^{k-1}f(\tau).
\]
\end{lemma}

\begin{proof}
By hypothesis,
\[
f(A\tau)=\beta (c\tau+d)^kf(\tau).
\]
The proof follows immediately from taking $\frac{d}{d\tau}$ of both sides of this equation.
\end{proof}

The next lemma is used in our construction of explicit examples in Section \ref{sec:exs1}.

\begin{lemma} \label{lemma:Ztransform} 
Let $Z(\tau)$ be a modular form for $\G$ of even weight $k$.  If $\gamma_s = \frac1{\sqrt s} \left(\begin{smallmatrix} 0 & -1 \\ s & 0 \end{smallmatrix}\right) \in \G$ for some real $s>0$, then for any real $r>0$, we have
\[
Z\left( \frac{i}{\sqrt{rs}}\right) = (-r)^{k/2} Z\left(i\sqrt{\frac rs} \right).
\]
\end{lemma}
\qed 

\begin{proof}
By the modularity of $Z(\tau)$, for any $\tau \in \hh$,
\[ Z(-1/s\tau) = (\sqrt{s}\tau)^kZ(\tau). \]
Setting $\tau=i\sqrt{\frac rs}$, gives the desired result.
\end{proof}

\begin{remark}\label{rmk:exist}
We note that $\PSL_2(\ZZ) \cong (2,3,\infty)$ contains $\gamma_1$, $\G_0(2)^+ \cong (2,4,\infty)$ contains $\gamma_2$, and $\G_0(3)^+ \cong (2,6,\infty)$ contains $\gamma_3$, with $\gamma_s$ as described in Lemma \ref{lemma:Ztransform}.
\end{remark}

\section{Proof of Theorem \ref{thm: general}}  \label{sec:pf1}

For convenience we use the notation $f'(\tau):=\frac{1}{2\pi i}\frac{df}{d\tau}(\tau)$ throughout this section.  

\begin{proof}[Proof of Theorem \ref{thm: general}.]

By hypothesis, we have for $\tau\in\hh$ that
\begin{equation}\label{eq:hypF}
Z(\gamma \tau)=\alpha (c\tau-a)^k Z(\tau).
\end{equation}
Thus by Lemma \ref{lem:slashderiv}, 
\[
Z'(\gamma\tau) = \alpha(c\tau-a)^{k+2}Z'(\tau) + \frac{\alpha ck}{2\pi i}(c\tau-a)^{k+1}Z(\tau).
\]
Plugging in $\tau_0=\frac{a}{c}+ \frac{i}{c\sqrt N}$ yields
\[
\frac{ck\sqrt{N}}{2\pi}\cdot Z(\tau_0) = Z'(\tau_0) + \alpha^{-1}i^{-k}N^{\frac{k+2}{2}}Z'(\gamma\tau_0).
\]
Dividing by $Z(\tau_0)$ and using \eqref{eq:hypF}, we obtain
\begin{equation}\label{eq:Z'tau0}
\frac{ck\sqrt{N}}{2\pi} = \frac{Z'(\tau_0)}{Z(\tau_0)} + N\frac{Z'(\gamma \tau_0)}{Z(\gamma \tau_0)}.
\end{equation}

On the other hand, changing the variable $\tau \mapsto N\tau$, we also have by hypothesis that 
\[
Z(\delta \cdot N \tau)=\beta Z(N\tau),
\]
so applying Lemma \ref{lem:slashderiv} with $c=0$ yields that 
\[
Z'\left(\delta \cdot N\tau \right) = \beta Z'\left(N\tau\right).
\]
Therefore, given $M_N(\tau)=\frac{Z(\tau)}{Z(N\tau)}$, we obtain

\begin{equation}\label{eq:M'}
\frac{M_N'(\tau)}{M_N(\tau)}=\frac{Z'(\tau)}{Z(\tau)}-N\frac{Z'(N\tau)}{Z(N\tau)}=\frac{Z'(\tau)}{Z(\tau)}-N\frac{Z'\left(\delta \cdot N\tau\right)}{Z\left(\delta \cdot N \tau\right)}.
\end{equation}

Note that $\gamma\cdot \tau_0 = \delta\cdot N \tau_0$ so from  \eqref{eq:hypF} we obtain
$$\beta Z(N\tau_0) = Z(\delta \cdot N\tau_0)=Z(\gamma \cdot \tau_0)=\alpha \frac{i^k}{N^{k/2}}Z(\tau_0),$$
and therefore $M_N(\tau_0)=\frac{\beta N^{k/2}}{\alpha i^k}$.  So \eqref{eq:M'} at $\tau_0$ becomes
\begin{equation}\label{eq:M'tau0}
\frac{\alpha i^k}{\beta N^{k/2}}M_N'(\tau_0)=\frac{Z'(\tau_0)}{Z(\tau_0)}-N\frac{Z'(\gamma \tau_0)}{Z(\gamma \tau_0)}.
\end{equation}

Using \eqref{eq:M'tau0} we can rewrite \eqref{eq:Z'tau0} as both
\begin{align}
\frac{ck\sqrt{N}}{2\pi} &= 2\frac{Z'(\tau_0)}{Z(\tau_0)} - \frac{\alpha i^k}{\beta N^{k/2}}M_N'(\tau_0), \label{eq:tau0} \\
\frac{ck\sqrt{N}}{2\pi} &= 2N\frac{Z'(\gamma \tau_0)}{Z(\gamma \tau_0)} + \frac{\alpha i^k}{\beta N^{k/2}}M_N'(\tau_0). \label{eq:gammatau0}
\end{align}

From our hypotheses, we have that $X' (\tau)= U(\tau)X(\tau)Z(\tau)$ for $\tau\in \hh$ and $Z(\tau)=X(\tau)^{\varepsilon_0}(1-X(\tau))^{\varepsilon_1}\sum_{j \ge 0} A_j X(\tau)^j$ for $\tau\in D$.  Using the fact that 
\[
\frac{X(\tau)}{X'(\tau)} \cdot \frac{1}{2\pi i}\frac{d}{d\tau} X(\tau)^{\varepsilon_0}(1-X(\tau))^{\varepsilon_1} = X(\tau)^{\varepsilon_0}(1-X(\tau))^{\varepsilon_1}\left(\varepsilon_0 + \varepsilon_1\frac{X(\tau)}{X(\tau)-1}\right),
\]
it follows that for $\tau \in D$,
\begin{equation}\label{eq:Z'/Z}
\frac{Z'(\tau)}{Z(\tau)}=U(\tau)X(\tau)^{\varepsilon_0}(1-X(\tau))^{\varepsilon_1}\sum_{j\geq 0}\left(j + \varepsilon_0 + \varepsilon_1\frac{X(\tau)}{X(\tau)-1}\right)A_j X(\tau)^j.
\end{equation}
On the other hand for $\tau\in D$, $M_N(\tau)$ can be expressed as a function of $X(\tau)$ and $X(N\tau)$, so 
\begin{equation}\label{eq:M_N'gen}
\frac{1}{2\pi i}\frac{dM_N}{d\tau}(\tau)=\frac{dM_N}{dX}(\tau) \cdot U(\tau)X(\tau)Z(\tau).
\end{equation}
Thus for $\tau\in D$,
\begin{equation}\label{eq:M_N'}
\frac{1}{2\pi i}\frac{dM_N}{d\tau}(\tau) = U(\tau) X(\tau)^{\varepsilon_0+1} (1-X(\tau))^{\varepsilon_1} \frac{dM_N}{dX}(\tau) \sum_{j \ge 0} A_j X(\tau)^j.
\end{equation}

Hence when $\tau_0\in D$ we can rewrite \eqref{eq:tau0} using \eqref{eq:Z'/Z} and \eqref{eq:M_N'} to obtain 
\begin{multline*}
\frac{ck\sqrt N}{2\pi} = 2U(\tau_0)X(\tau_0)^{\varepsilon_0}(1-X(\tau_0))^{\varepsilon_1}\sum_{j\geq 0}\left(j + \varepsilon_0 + \varepsilon_1\frac{X(\tau_0)}{X(\tau_0)-1}\right)A_j X(\tau_0)^j \\
- \frac{\alpha i^k}{\beta N^{k/2}} U(\tau_0)X(\tau_0)^{\varepsilon_0+1} (1-X(\tau_0))^{\varepsilon_1} \left(\frac{dM_N}{dX}\right)\mid_{X=X(\tau_0)}\sum_{j\geq 0}  A_j X(\tau_0)^j, 
\end{multline*}
which yields our first identity. Similarly when $\gamma\tau_0\in D$ we rewrite \eqref{eq:gammatau0} using \eqref{eq:Z'/Z} and \eqref{eq:M_N'gen} to obtain
\begin{multline*}
\frac{ck\sqrt N}{2\pi} = 2NU(\gamma\tau_0)X(\gamma\tau_0)^{\varepsilon_0}(1-X(\gamma\tau_0))^{\varepsilon_1}\sum_{j\geq 0} \left(j + \varepsilon_0 + \varepsilon_1\frac{X(\gamma\tau_0)}{X(\gamma\tau_0)-1}\right) A_j X(\gamma\tau_0)^j \\
+ \frac{\alpha i^k}{\beta N^{k/2}}\left(\frac{dM_N}{dX}\right)\mid_{X=X(\tau_0)} \cdot U(\tau_0)X(\tau_0)Z(\tau_0).
\end{multline*}
Since by \eqref{eq:hypF} we have $Z(\tau_0)= \frac{N^{k/2}}{\alpha i^k}Z(\gamma \tau_0)$, this becomes
\begin{multline*}
\frac{ck\sqrt N}{2\pi} = 2NU(\gamma\tau_0)X(\gamma\tau_0)^{\varepsilon_0}(1-X(\gamma\tau_0))^{\varepsilon_1}\sum_{j\geq 0} \left(j + \varepsilon_0 + \varepsilon_1\frac{X(\gamma\tau_0)}{X(\gamma\tau_0)-1}\right) A_j X(\gamma\tau_0)^j \\ 
+ \frac{1}{\beta} U(\tau_0)X(\tau_0) \left(\frac{dM_N}{dX}\right)\mid_{X=X(\tau_0)} X(\gamma\tau_0)^{\varepsilon_0}(1-X(\gamma\tau_0))^{\varepsilon_1} \sum_{j\geq 0}  A_j X(\gamma \tau_0)^j,
\end{multline*}
which yields our second identity.
\end{proof}

We conclude this section with a lemma that will be useful when we compute examples in the following sections.

\begin{lemma} \label{lem:M_N}
Suppose $X(\tau), Z(\tau), U(\tau)$ satisfy the conditions in Theorem \ref{thm: general}.  Then, writing $X:=X(\tau)$ and $Y:=X(N\tau)$ gives
\[
M_N = N \frac{dX}{dY}\cdot \frac{Y}{X} \cdot \frac{U(N\tau)}{U(\tau)}.
\]
\end{lemma}

\begin{proof}
Since $M_N$ is a function of $X$ and $Y$,
\begin{multline*}
\frac{dY}{dX}=\frac{dY}{d\tau}\cdot \frac{d\tau}{dX} =N\frac{dX}{d\tau}(N\tau)\cdot \frac{d\tau}{dX}
= N \frac{\frac{dX}{d\tau}(N\tau)}{\frac{dX}{d\tau}}
= N\frac{YU(N\tau)Z(N\tau)}{XU(\tau)Z(\tau)} = \frac{N}{M_N}\frac{Y}{X}\frac{U(N\tau)}{U(\tau)},
\end{multline*}
which gives the result.
\end{proof}

\section{Examples of Theorem \ref{thm: general} of first type} \label{sec:exs2}
In this section we obtain six different Ramanujan-Sato series for $1/ \pi$ as examples of Theorem \ref{thm: general}. One is already available in the literature and the others are new according to our knowledge.

To construct the examples in this section we use the following corollary for modular forms on certain groups where the cusp $i\infty$ has width $1$, which follows immediately from Theorem \ref{thm: general} with $h=1$ and $\beta=1$. 

\begin{corollary} \label{cor:1.3}
Let $\Gamma$ be a discrete subgroup of $\SL_2(\RR)$ commensurable with $\SL_2(\ZZ)$ such that $\left(\begin{smallmatrix}1&1\\0&1\end{smallmatrix}\right) \in \Gamma$ (i.e. the cusp $i\infty$ of $\G$ has width $1$), and let $X(\tau)$ be a Hauptmodul of $\Gamma$.  Let $Z(\tau)$ be a weight-$k$ modular form for $\Gamma$ such that $\frac{1}{2\pi i}\frac{dX}{d\tau} = U(\tau)X(\tau)Z(\tau)$ and when $\tau\in D$, a domain of $\hh$, $Z(\tau)= X(\tau)^{\varepsilon_0}(1-X(\tau))^{\varepsilon_1} \sum\limits_{j=0}^\infty A_jX(\tau)^j$ for $\varepsilon_0,\varepsilon_1\in \RR$, $A_j\in \CC$.  Further assume there exist $\alpha \in \CC$ and $\gamma = \left(\begin{smallmatrix}a&b\\c&-a\end{smallmatrix}\right) \in \SL_2(\RR)$ such that 
\[
\left( Z \vert_k  \gamma \right) (\tau) = \alpha Z (\tau).
\]
Set $M_N(\tau):=Z(\tau)/Z(N\tau)$ for $N\in \NN$ satisfying $\frac ac(1-N)\in \ZZ$, and let $\tau_0=\frac{a}{c}+ \frac{i}{c\sqrt N}$. Then if $\gamma\tau_0 =\frac ac + \frac{i\sqrt N}{c} \in D$, we have
\[
\frac{ck\sqrt N}{2\pi}=X(\gamma\tau_0)^{\varepsilon_0}(1-X(\gamma\tau_0))^{\varepsilon_1}\sum_{j\geq 0} (b_Nj+a_N) A_j X(\gamma\tau_0)^j,
\]
where
\begin{align*}
b_N &= 2NU(\gamma\tau_0), &  \\
a_N &= 2N U(\gamma\tau_0)\left(\varepsilon_0+\varepsilon_1\frac{X(\gamma\tau_0)}{X(\gamma\tau_0)-1}\right) + U(\tau_0)X(\tau_0) \left(\frac{dM_N}{dX}\right) \left|_{X=X(\tau_0)}. \right.
\end{align*}
\end{corollary}  

Note that the arithmetic triangle groups $(m,\infty,\infty) \cong \G_0(m)$ for $m\in\{2,3\}$ and $(\infty,\infty,\infty) \cong \G_0(4)$ contain the element $\left(\begin{smallmatrix}1&1\\0&1\end{smallmatrix}\right)$. In this section we construct examples of Corollary~\ref{cor:1.3} for these groups. We choose the Hauptmodul $t_m$ for $\Gamma_0(m)$ with $m\in\{2,3\}$ and $t_\infty$ for $\Gamma_0(4)$ given in Table~\ref{tab:Hauptmoduln}.  Furthermore, from Table~\ref{tab:gens} we have that the space of modular forms of weight $2$ for $\Gamma_0(m)$ for $m=2,3,4$ is generated by $Z_2$, $Z_3$ and $Z_{\infty}$, respectively; these are given by
\begin{equation}\label{Zfort2}
    Z_m(\tau)=(1-t_m)^{1-\frac 1m}\pFq21{ \frac12- \frac1{2m} &  \frac12- \frac1{2m} }{&1}{t_m(\tau)}^2,
\end{equation}
and 
\begin{equation}\label{Zfort2inf}
    Z_\infty(\tau)=\pFq21{ \frac12&  \frac12}{&1}{t_\infty(\tau)}^2.
\end{equation}
Write $f':= \frac{1}{2\pi i}\frac{df}{d\tau}$. 
We next compute each $U_m$ for $t_m$ with $m\in\{2,3\}$ so that 
\begin{equation}
t_m' = t_mZ_mU_m.
\end{equation}
From Table~\ref{tab:Hauptmoduln}, the Hauptmodul $t_{m}$ for $m \in \{ 2, 3\}$ can be written as
\begin{equation}
    \label{eqn: tmalpha}
    t_{m}(\tau)=-\alpha\frac{\eta\left(m\tau\right)^{k_m}}{\eta(\tau)^{k_m}},
\end{equation} 
where 
$\alpha = m^{\frac{k_m}{4}}$ and  $k_m=\frac{24}{m-1}$.
Using the property $\eta'/\eta=\frac{1}{24} E_2$, we  get that
\begin{align*}
    \frac{t_m'}{t_m}(\tau)= k_m \left( m\frac{\eta'(m\tau)}{\eta(m\tau)}-\frac{\eta'(\tau)}{\eta(\tau)}\right)= \left( \frac{-1}{m-1}\right) E_{2, m}.
\end{align*}
Moreover, since $E_{2,m}$ is a modular form of weight $2$ for $\Gamma_0(m)$ by (\ref{eqn: E2N}), and $Z_m$ generates the space of weight $2$ modular forms for $\Gamma_0(m)$ as seen in Table \ref{tab:gens}, Theorem \ref{thm:Sturm} gives $Z_m(\tau)=\left( \frac{1}{1-m}\right) E_{2, m}$.  Therefore for $m=2,3$,
\begin{equation}\label{U1fort2t3}
    t_{m}' =t_{m}Z_m, \mbox{ and } U_m =1. 
\end{equation}
In order to apply Corollary \ref{cor:1.3} in these cases we need $\frac{dM_N}{dX}$ where $X(\tau)=t_m(\tau)$ and $Y(\tau)=t_m(N\tau)$.  Observe from Lemma \ref{lem:M_N} that when $U=1$, we have
\[
M_N=N\frac{dX}{dY}\cdot\frac{Y}{X}.
\]
Thus taking the derivative with respect to $X$ yields
\begin{equation}\label{eq:dM_NU=1}
\frac{dM_N}{dX} = N\left(\frac{d\left(\frac{dX}{dY}\right)}{dX}\cdot \frac{Y}{X} + \frac{1}{X} - \frac{dX}{dY}\cdot\frac{Y}{X^2} \right),
\end{equation}   
which will be useful in subsequent subsections.
 
The computation of $U$ and $\frac{dM_N}{dX}$ for the group $\Gamma_0(4)$ is given in \S\ref{sec: G0(4)} where we obtain an example for this group. 
\subsection{\texorpdfstring{$\Gamma_0(2)\cong (2,\infty,\infty)$}{}} 
As in Table~\ref{tab:Hauptmoduln} we choose the Hauptmodul $X(\tau)=t_2(\tau)=-64\frac{\eta(2\tau)^{24}}{\eta(\tau)^{24}}$, and for the corresponding domain $D$ we choose the intersection of $\{\tau\in \hh:\, |X(\tau)|<1\}$ and the following fundamental domain $FD$ of $\G_0(2)$:
$$
 FD=\{\tau\in \hh:\, |\mbox{Re}(\tau)|\leq 1/2,\, |\tau-1/2|>1/2, \, |\tau+1/2|>1/2\}.
$$
Recall that $Z=Z_2$ from \eqref{Zfort2}.  It follows from (\ref{eq:Clausen}) that
\begin{equation}\label{Zfort21}
    Z(\tau)=(1-X(\tau))^{\frac 12}\pFq21{ \frac14 & \frac14}{&1}{X(\tau)}^2=(1-X(\tau))^{\frac 12}\pFq32{ \frac12 & \frac 12 &\frac 12}{&1&1}{X(\tau)}.
\end{equation}
Thus,
\begin{equation}\label{Zfort22}
Z(\tau) = (1-X(\tau))^{\frac 12}\sum\limits_{j=0}^\infty \frac{\left(\frac12\right)_j\left(\frac12\right)_j\left(\frac12\right)_j}{(j!)^3} X^j(\tau).
\end{equation}
For Examples~\ref{sect4ex1} and \ref{sect4ex2} below, we use $\gamma=w_2=\frac1{\sqrt 2}\left(\begin{smallmatrix} 0 & -1 \\ 2 & 0 \end{smallmatrix}\right)$. To find the transformation property of $Z$ with respect to $w_2$, we first note that Theorem~\ref{thm:Sturm} implies
$$Z(\tau)=\frac{1}{2}(\theta_3(\tau)^4+\theta_4(\tau)^4).$$
Also, using the transformation formula $\eta(-1/\tau)=\sqrt{-i\tau}\eta(\tau)$ \cite[p.20]{DiamondShurman} and (\ref{eqn: thetaeta}), we have 
\begin{align*}
    \theta_3^4\left(\frac{-1}{2\tau}\right)&=-4\tau^2\theta_3^4(2\tau),\\
      \theta_4^4\left(\frac{-1}{2\tau}\right)&=-4\tau^2\cdot 16\frac{\eta^8(4\tau)}{\eta^4(2\tau)}=-4\tau^2 \theta_2^4(2\tau).
\end{align*}
Using the last two equalities in (\ref{eqn: thetaE2}), we thus obtain the following transformation property
\begin{equation}\label{Ztransw2}
    (Z|w_2)(\tau)=-Z(\tau), \mbox{ i.e., } Z(-1/2\tau)=-2\tau^2 Z(\tau).
\end{equation}

We see that $Z$ and $X$ above satisfy the conditions in Corollary~\ref{cor:1.3} with $U=1$, $\gamma=w_2$, $k=2$, $\varepsilon_0=0$, $\varepsilon_1=\frac 12$, $\alpha=-1$, and any $N\in \NN$.

For Examples~\ref{sect4ex3} and \ref{sect4ex4} below, we consider $\gamma=\smat{1}{-1}{2}{-1}$. Since $Z=Z_2$ is a weight $2$ modular form on $\Gamma_0(2)$, we have the following transformation property
\begin{equation}\label{Ztrans2}
(Z|\gamma)(\tau)=Z(\tau), \mbox{ i.e., } Z\left(\frac{\tau-1}{2\tau-1}\right)=(2\tau-1)^2Z(\tau). 
\end{equation}
We see that $Z$ and $X$ above satisfy conditions in Corollary~\ref{cor:1.3} with $U=1$, $\gamma=\smat{1}{-1}{2}{-1}$, $k=2$, $\varepsilon_0=0$, $\varepsilon_1=\frac 12$, and $\alpha=1$. Note that, for this choice of $\gamma$, we need to choose $N$ such that $\frac{1-N}{2}\in \ZZ$. 
 
We now proceed with examples for the group $\Gamma_0(2)\cong(2,\infty,\infty)$.

\begin{example}\label{sect4ex1}
Let $N=3$, $\gamma=w_2$ above, and $\tau_0=i/\sqrt 6$. Using the values of $\eta(i/\sqrt{6})$, $\eta(i\sqrt{6})$, $\eta(i\sqrt{ 2/3})$ and $\eta(i\sqrt{3/2})$ from Table~\ref{tab:etavalues} we get
 \begin{equation*}
     \begin{split}
        X(i/\sqrt{6})&=-17-12\sqrt{2},\\
         Y(i/\sqrt{6})=X(i\sqrt{3/2})&=-17+12\sqrt2.
    \end{split}
\end{equation*}
Using \eqref{eq:dM_NU=1} and the polynomial relationship $\Phi_3(X,Y)=0$ between $X(\tau)=t_2(\tau)$ and $Y(\tau)=t_2(\gamma\tau)$ given in Lemma~\ref{modulart2N3}, we obtain
\begin{equation*}
\left(\frac{dM_3}{dX}\right)\left|_{X=X(i/\sqrt{6})}\right.=12-\frac{17}{\sqrt{2}}.
\end{equation*}
From Corollary~\ref{cor:1.3} and \eqref{Zfort22} we then have
\[
\frac{\sqrt {6}}{\pi} = (1-X(i\sqrt{3/2}))^{1/2} \sum_{j\geq 0} (6j+a_3) \frac{\left(\frac12\right)_j\left(\frac12\right)_j\left(\frac12\right)_j}{(j!)^3} X(i\sqrt{3/2})^j,
\]
where
\begin{align*}
a_3&= 3\left(\frac{X\left(i\sqrt{\frac{3}{2}}\right)}{X\left(i\sqrt{\frac{3}{2}}\right)-1}\right) 
+X(i/\sqrt{6}) \left(\frac{dM_3}{dX}\right)\left|_{X=X(i/\sqrt{6})} \right. =\frac 32 - \frac 1{\sqrt{2}}.
\end{align*}
Finally we have 
\[
\frac{2}{\pi} = (\sqrt 2-1) \sum_{j\geq 0} \left(12j+3 -\sqrt{2}\right) \frac{\left(\frac12\right)_j\left(\frac12\right)_j\left(\frac12\right)_j}{(j!)^3} \left(12\sqrt2-17\right)^j.
\]
\end{example}
\begin{example}\label{sect4ex2}
Let $N=5$, $\gamma=w_2$, and $\tau_0=i/\sqrt{10}$. Using the values of $j(i/\sqrt{10})$ and $j\left(i\sqrt{\frac 52}\right)$ from Table~\ref{tab:jvalues} and 
 the relationship $j=\frac{64(4X-1)^3}{X}$ from  (\ref{eqn: jt2reln}), we get the values
  \begin{equation*}
     \begin{split}
        X(i/\sqrt{10})&=-161-72\sqrt{5},\\
         X\left(i\sqrt{\frac{5}{2}}\right)&=-161+72\sqrt{5}.
    \end{split}
\end{equation*}
Using the polynomial relationship $\Phi_5(X,Y)=0$ between $X(\tau)=t_2(\tau)$ and $Y(\tau)=t_2(\gamma\tau)$ from Lemma~\ref{modulart2N3}, we obtain
\begin{equation*}
   \frac{dM_5}{dX}|_{X=X(i/\sqrt{10})}=\frac{1440-644\sqrt{5}}{9}. 
\end{equation*}
From Corollary~\ref{cor:1.3} and \eqref{Zfort22} we then have
\[
\frac{\sqrt {10}}{\pi} = \left(1-X\left(i\sqrt{\frac{5}{2}}\right)\right)^{1/2} \sum_{j\geq 0} (10j+a_5) \frac{\left(\frac12\right)_j\left(\frac12\right)_j\left(\frac12\right)_j}{(j!)^3} X\left(i\sqrt{\frac{5}{2}}\right)^j,
\]
where
\begin{align*}
a_5&= 5\left(\frac{X\left(i\sqrt{\frac{5}{2}}\right)}{X\left(i\sqrt{\frac{5}{2}}\right)-1}\right) 
+X(i/\sqrt{10}) \left(\frac{dM_5}{dX}\right)\left|_{X=X(i/\sqrt{10})} \right. =\frac{5}{2}-\frac{2 \sqrt{5}}{3}.
\end{align*}
So, finally we have 
\[
\frac{2\sqrt {5}}{\pi} = (\sqrt 5-2) \sum_{j\geq 0} \left(60j+15-4 \sqrt{5}\right) \frac{\left(\frac12\right)_j\left(\frac12\right)_j\left(\frac12\right)_j}{(j!)^3} \left(72\sqrt{5}-161\right)^j.
\]
\end{example}
\begin{example}\label{sect4ex3}
Let $N=3$, $\gamma=\smat{1}{-1}{2}{-1}$, and  $\tau_0=\frac{1+i\sqrt {3}}{2}$. Using the values $\eta\left(\frac{-1+i\sqrt{3}}{2}\right)$, $\eta(i\sqrt{3})$ from Table~\ref{tab:etavalues}, and the transformation formula  $\eta(\gamma\tau)^{24}=(c\tau+d)^{12}\eta(\tau)^{24}$ \cite[p. 20]{DiamondShurman}, we have $$X(\tau_0)=X(\gamma\tau_0)=\frac 14.$$ 
Using the polynomial relationship $\Phi_3(X,Y)=0$ between $X(\tau)=t_2(\tau)$ and $Y(\tau)=t_2(\gamma\tau)$ from Lemma~\ref{modulart2N3} and \eqref{eq:dM_NU=1} we obtain
\begin{equation}
\left(\frac{dM_3}{dX}\right)\left|_{X=X(\tau_0)} \right.=8.
\end{equation}
Corollary~\ref{cor:1.3} and \eqref{Zfort22} yield the series
\begin{equation*}
\frac{2\sqrt 3}{\pi} =(1-X(\gamma\tau_0))^{\frac 12} \sum\limits_{j\geq0} (6j +  a_3) \frac{\left(\frac12\right)_j\left(\frac12\right)_j\left(\frac12\right)_j}{(j!)^3} X(\gamma\tau_0)^j,
\end{equation*}
where
\begin{align*}
a_3&= 3\left(\frac{X(\gamma\tau_0)}{X(\gamma\tau_0)-1}\right) 
+X(\tau_0) \left(\frac{dM_3}{dX}\right)\left|_{X=X(\tau_0)} \right.=1. 
\end{align*}
Hence, we obtain
\begin{align*}
\frac{4}{\pi} &= \sum\limits_{j=0}^\infty (1 + 6j)\frac{\left(\frac12\right)_j\left(\frac12\right)_j\left(\frac12\right)_j}{(j!)^3} (1/4)^j.
\end{align*}
This series is one of the well-known Ramanujan series in \cite{Ramanujan} for $1/\pi$, which also arises from Chan, Chan, and Liu \cite[(1.1)]{CCL}).
\end{example}

\begin{example}\label{sect4ex4}
Let $N=5$, $\gamma=\smat{1}{-1}{2}{-1}$, and  $\tau_0=\frac{1+i\sqrt {5}}{2}$. Then, by using the values of $j(\tau_0)$, $j(\gamma\tau_0)$ from Table~\ref{tab:jvalues} and the relationship $j=\frac{64(4X-1)^3}{X}$ from \eqref{eqn: jt2reln}, we get $$X(\tau_0)=X(\gamma\tau_0)=9-4\sqrt{5}.$$ 
Next, using the polynomial relationship $\Phi_5(X,Y)=0$ between $X(\tau)=t_2(\tau)$ and $Y(\tau)=t_2(\gamma\tau)$ from Lemma~\ref{modulart2N3}, we obtain
\begin{equation*}
\left(\frac{dM_5}{dX}\right)\left|_{X=X(\tau_0)} \right.=15+ \frac{27\sqrt{5}}{4}.
\end{equation*}
Then, from Corollary~\ref{cor:1.3} and \eqref{Zfort22}, the series is of the form
\begin{equation*}
\frac{2\sqrt 5}{\pi} =(1-X(\gamma\tau_0))^{\frac 12} \sum\limits_{j\geq0} (10j + a_5) \frac{\left(\frac12\right)_j\left(\frac12\right)_j\left(\frac12\right)_j}{(j!)^3} X(\gamma\tau_0)^j,
\end{equation*}
where
    \begin{align*}
        a_5 &= 5\left(\frac{X(\gamma\tau_0)}{X(\gamma\tau_0)-1}\right) + X(\tau_0) \left(\frac{dM_5}{dX}\right)\left|_{X=X(\tau_0)} \right. =\frac{5 - \sqrt{5}}{2}.
 \end{align*}
So, we finally obtain 
\begin{equation*}
\frac{2\sqrt{5}}{\pi} =(\sqrt{5}-2)^{1/2} \sum\limits_{j\geq0} \left(20j +5 - \sqrt{5}\right) \frac{\left(\frac12\right)_j\left(\frac12\right)_j\left(\frac12\right)_j}{(j!)^3} (9-4\sqrt{5})^j.
\end{equation*}
\end{example}

\subsection{\texorpdfstring{$\G_0(3) \cong (3,\infty,\infty)$}{}} 
As in Table~\ref{tab:Hauptmoduln} we choose the Hauptmodul $$X(\tau)=t_3(\tau)=-27 \frac{\eta(3\tau)^{12}}{\eta(\tau)^{12}},$$
and for the corresponding domain $D$ we choose the intersection of $\{\tau\in \hh:\, |X(\tau)|<1\}$ and the following fundamental domain $FD$ of $\G_0(3)$
$$
 FD=\{\tau\in \hh:\, |\mbox{Re}(\tau)|\leq 1/2,\, |\tau-1/3|>1/3, \, |\tau+1/3|>1/3\}.
$$
Recall $Z=Z_3$ as given in \eqref{Zfort2}. From the hypergeometric product formula \cite[Theorem~2.3]{Kaiblinger}, we get that 

\begin{equation*}
  Z(\tau) = \sum\limits_{j=0}^\infty A_j X(\tau)^j,
\end{equation*}
where 
\begin{equation}\label{Zfort3}
    A_j = \frac{\left(\frac{1}{3}\right)_j^2}{j!^2} \sum\limits_{k=0}^j \frac{\left(-j\right)_k^2 \left(\frac{1}{3}\right)_k^2}{k!^2 \left(\frac{2}{3}-j\right)_k^2}.
\end{equation}
Recalling \eqref{U1fort2t3}, we see that $Z$ and $X$ above satisfy the conditions in Corollary~\ref{cor:1.3} with $U=1$, $k=2$, $\varepsilon_0=0$, $\varepsilon_1=\frac 23$, and $\alpha=1$. 
\begin{example}\label{sect4ex5}
Consider $N=2$, $\gamma=w_3=\frac1{\sqrt 3}\left(\begin{smallmatrix} 0 & -1 \\ 3 & 0 \end{smallmatrix}\right)$, and $\tau_0=\frac{i}{\sqrt 6}$. Using the values $\eta\left(i\sqrt{3/2} \right)$, $\eta\left(i/\sqrt{6} \right)$, $\eta\left(i\sqrt{6} \right)$ and $\eta\left(i\sqrt{2/3} \right)$ from Table~\ref{tab:etavalues}, we have 
  \begin{equation*}
     \begin{split}
        t_3\left(\frac{i}{\sqrt{6}}\right)&=-3-2\sqrt{2},\\
         t_{3}\left(i\sqrt{\frac 23}\right)&=-3+2\sqrt{2}.
    \end{split}
\end{equation*}
Using the polynomial relationship $\Phi_2(X,Y)=0$ between $X(\tau)=t_3(\tau)$ and $Y(\tau)=t_3(\gamma\tau)$ from Lemma~\ref{modulart3N2} we get
\begin{equation*}
\left(\frac{dM_2}{dX}\right)\left|_{X=X(\tau_0)} \right.=\frac{4-3\sqrt{2}}{3}. 
\end{equation*}
Then, from Corollary~\ref{cor:1.3} and \eqref{Zfort22}, we get the following series
\[
\frac{\sqrt {6}}{\pi} = (1-X(\gamma \tau_0))^{2/3}  \sum_{j\geq 0} (4j+a_2) A_j X(\gamma \tau_0)^j,
\]
where  $A_j$ is given in \eqref{Zfort3} and
\begin{equation*}
\begin{split}
a_2&=4  \left(\frac 23 \frac{X(\gamma\tau_0)}{X(\gamma\tau_0)-1}\right)
+ X(\tau_0) \left(\frac{dM_2}{dX}\right)\left|_{X=X(\tau_0)} \right.=\frac{4-\sqrt{2}}{3}.
\end{split}
\end{equation*}
Then, finally we obtain 
$$\frac{3\sqrt{6}}{\pi}=(4-2\sqrt{2})^{2/3}\sum_{j\ge 0} (12j+4-\sqrt{2})A_j (2\sqrt{2}-3)^j,$$
where  $A_j$ is given in \eqref{Zfort3}.
\end{example}

\subsection{\texorpdfstring{$ \G_0(4)\cong (\infty,\infty,\infty)$}{}}
\label{sec: G0(4)}
By Table~\ref{tab:Hauptmoduln}, we choose the Hauptmodul
\begin{equation}\label{eq:Xinfinity}
X(\tau)= t_\infty(\tau)=16\eta(\tau)^8\eta(4\tau)^{16}/\eta(2\tau)^{24},
\end{equation}
which can also be written $\lambda(2\tau)$, where $\lambda(\tau)$ is the modular lambda function introduced in (\ref{defn: lambda}).  For the corresponding fundamental domain, we use
\[
\{ \tau\in \hh :\, |\text{Re}(\tau)|\leq 1, |\tau-1/4|>1/4, |\tau-3/4|>1/4\}.
\]
Then using \eqref{Zfort2inf} and \eqref{eqn: thetahypgeom}, $Z$ is a weight $2$ modular form for $\G_0(4)$ given by
\[
Z(\tau)= \pFq21{ \frac12 & \frac12}{&1}{X(\tau)}^2=\theta_3(2\tau)^4.
\]
Using the hypergeometric product formula \cite[Theorem~2.3]{Kaiblinger}, 
\begin{equation*}
  Z(\tau) = \sum\limits_{j=0}^\infty A_j X(\tau)^j,
\end{equation*}
where 
\begin{equation}\label{Ajfortinf}
    A_j = \frac{\left(\frac{1}{2}\right)_j^2}{j!^2}\sum\limits_{k=0}^j \frac{\left(j\right)_k^2\left(\frac{1}{2}\right)_k^2}{k!^2\left(\frac{1}{2}-j\right)_k^2}.
\end{equation}
As before, we write $f':= \frac{1}{2\pi i}\frac{df}{d\tau}$. 
We next compute $U$ so that 
\begin{equation}
X' = XZU.
\end{equation}
Differentiating $X(\tau)=16\eta(\tau)^8\eta(4\tau)^{16}/\eta(2\tau)^{24}$ and using the classical fact that $\eta'=\frac{1}{24}\eta E_2$, we compute that
\[
X' = (1-X)XZ.
\]
Here we needed to use the identity $E_{2,2}(\tau)=-(\theta_2(2\tau)^4+\theta_3(2\tau)^4)$ from (\ref{eqn: thetaE2}). Thus $$U=1-X.$$
From Lemma \ref{lem:M_N}, we obtain the following by
differentiating $M_N$ with respect to $X$
\[
\frac{dM_N}{dX}= N \frac{Y}{X}\frac{1-Y}{1-X}\left( \frac{d\left(\frac{dX}{dY}\right)}{dX}+ \frac{dX}{dY}\frac{X\left(\frac{dY}{dX}\right)-Y}{XY}+
    \frac{dX}{dY}\frac{(1-X)\left(-\frac{dY}{dX}\right)+(1-Y)}{(1-Y)(1-X)}\right).
\]

We choose $\gamma=\left(\begin{smallmatrix}0 & -1/2 \\ 2 & 0\end{smallmatrix}\right)$. Using transformation properties of theta functions we can show for $\gamma$ that
\[
Z(-1/4\tau) = \theta_3(-1/2\tau)^4 = -4\tau^2\cdot\theta_3(2\tau)^4=-4\tau^2 Z(\tau).
\]
Thus $Z$ and $X$ above satisfy the conditions in Corollary~\ref{cor:1.3} with $U=1-X$, $\gamma=\left(\begin{smallmatrix}0 & -1/2 \\ 2 & 0\end{smallmatrix}\right)$, $k=2$, $\varepsilon_0=0$, $\varepsilon_1=0$, and $\alpha=-1$.

\begin{example} Let $N=2$, $\gamma=\left(\begin{smallmatrix}0 & -1/2 \\ 2 & 0\end{smallmatrix}\right)$, $\tau_0=i/2\sqrt{2}$. 
Using the value $\lambda(\sqrt{2}i) = (\sqrt{2}-1)^2$ \cite[(4.6.10)]{BB}, we obtain
\begin{align*} 
X(i/\sqrt{2}) &= 3-2\sqrt{2}, \\
U(i/\sqrt{2}) &= 2(\sqrt{2}-1).
\end{align*}
Moreover, the transformation property of $\lambda$ \eqref{eqn: lambdatransform} yields that
\begin{align*} 
X(i/2\sqrt{2}) &= \lambda(i/\sqrt{2})=1-(\sqrt{2}-1)^2 = 2(\sqrt{2}-1),\\
U(i/2\sqrt{2}) &= 3-2\sqrt{2}. 
\end{align*}
Next, using the polynomial relationship $\Phi_2(X,Y)=0$ satisfied by $X(\tau)=t_{\infty}(\tau)$ and $Y(\tau)=t_{\infty}(\gamma\tau)$ from Lemma~\ref{modularpolytoo} we get
\[
\frac{dM_2}{dX}\mid_{X=X(i/2\sqrt{2})} = \sqrt{2}+2.  
\]
Thus applying Corollary~\ref{cor:1.3} yields that
\[
\frac{\sqrt{2}}{\pi} = \sum_{j\geq 0} (4(\sqrt{2}-1)j-4+3\sqrt{2}) A_j (3-2\sqrt{2})^j,
\]
where $A_j$ is given in \eqref{Ajfortinf}. 
\end{example}

\section{Examples of Theorem \ref{thm: general} of second type} \label{sec:exs1}
In this section we obtain five additional Ramanujan-Sato series for $1/ \pi$ as examples of Theorem \ref{thm: general}.
To construct the examples we use the following corollary for modular forms on certain groups where the cusp $i\infty$ has general width $h$, which follows immediately from Theorem \ref{thm: general} with $\gamma=\gamma_s$, $\delta=I$, $\alpha=1$, $\beta=1$, and $N\in\NN$ is arbitrary.

\begin{corollary}\label{cor:1.4}
Let $\Gamma$ be a discrete subgroup of $\SL_2(\RR)$ commensurable with $\SL_2(\ZZ)$ such that $\gamma_s = \frac1{\sqrt s}\left(\begin{smallmatrix} 0 & -1 \\ s & 0 \end{smallmatrix}\right) \in \Gamma$ for some real $s>0$, and let $h$ be the width of the cusp $i\infty$ of $\G$.  Let $X(\tau)$ be a Hauptmodul of $\Gamma$ and $Z(\tau)$ a weight-$k$ modular form for $\G$ such that $\frac{h}{2\pi i}\frac{dX}{d\tau} = U(\tau)X(\tau)Z(\tau)$ and when $\tau\in D$, a domain of $\hh$, $Z(\tau)= X(\tau)^{\varepsilon_0}(1-X(\tau))^{\varepsilon_1} \sum\limits_{j=0}^\infty A_jX(\tau)^j$ for $\varepsilon_0,\varepsilon_1\in \RR$, $A_j\in \CC$. Set  $M_N(\tau):=Z(\tau)/Z(N\tau)$ for $N\in \NN$ and let $\tau_0=\frac{i}{\sqrt{sN}}$. Then if $\tau_0\in D$,

\[
\frac{kh\sqrt{sN}}{2\pi} = U(i/\sqrt{sN})X(i/\sqrt{sN})^{\varepsilon_0}\left(1-X(i/\sqrt{sN})\right)^{\varepsilon_1} \sum\limits_{j\geq0} (2j + a_N) A_j X\!\left(i/\sqrt{sN}\right)^j,
\]
where
\[
a_N = 2\left(\varepsilon_0+\varepsilon_1\frac{X(i/\sqrt{sN})}{X(i/\sqrt{sN})-1}\right) -{\frac{ i^k}{ N^{k/2}}} X(i/\sqrt{sN}) \left(\frac{dM_N}{dX}\right)\left|_{X =X(i/\sqrt{sN})}. \right.
\]
Alternatively if $\gamma\tau_0\in D$,
\[
\frac{kh\sqrt{sN}}{2\pi} = X(i\sqrt{N}/\sqrt{s})^{\varepsilon_0}(1-X(i\sqrt{N}/\sqrt{s}))^{\varepsilon_1} \sum_{j\geq 0} (b_N'j+a'_N) A_j X(i\sqrt{N}/\sqrt{s})^j,
\]
where
\begin{align*}
b'_N = 2N U(i\sqrt{N}/\sqrt{s}), & \\
a'_N = 2N U(i\sqrt{N}/\sqrt{s}) & \left(\varepsilon_0+\varepsilon_1\frac{X(i\sqrt{N}/\sqrt{s})}{X(i\sqrt{N}/\sqrt{s})-1}\right) \\
& + U(i/\sqrt{sN})X(i/\sqrt{sN}) \left(\frac{dM_N}{dX}\right)\left|_{X=X(i/\sqrt{sN})}. \right.
\end{align*}
\end{corollary}

From Remark \ref{rmk:exist}, we know that the arithmetic triangle groups $(2,m,\infty)$ for $m\in\{3,4,6\}$ contain the element $\gamma_s$ for $s=\lfloor \frac m2 \rfloor$, respectively.  In this section we construct examples of Corollary \ref{cor:1.4} for these groups. 

Fix $m\in \{3,4,6\}$ and let $\G_m=(2,m,\infty)$. Then the width of the cusp $i\infty$ of $\G_m$ is $h=1$, and we have seen in Tables \ref{tab:Hauptmoduln} and \ref{tab:gens}   that $t_{2,m}$ is a Hauptmodul for $\G_m$  and that the space of modular forms of weight $4$ for $\G_m$ is generated by $Z_m(\tau)$.  Using Clausen's formula \eqref{eq:Clausen},

\begin{align*}
Z_m(\tau) := \pFq21{\frac14-\frac{1}{2m} & \frac14 + \frac{1}{2m}}{&1}{t_{2,m}(\tau)}^4 = \pFq32{ \frac12 & \frac12-\frac{1}{m} & \frac12+\frac{1}{m} }{&1&1}{t_{2,m}(\tau)}^2.
\end{align*}
Thus we see that
\begin{equation}
\label{eq: Zm as series of t}
    Z_m(\tau) = \sum_{j=0}^\infty A_{m,j} t_{2,m}(\tau)^j,
\end{equation}
where by the hypergeometric product formula \cite[Theorem~2.3]{Kaiblinger}, we have\footnote{Note that the $_6F_5$ series arising from the formula in \cite{Kaiblinger} naturally truncates at $j$.}
\begin{equation}
\label{eq: A_j,m general}
    A_{m,j} = \frac{\left(\frac{1}{2}\right)_j\left(\frac{1}{2}-\frac{1}{m}\right)_j\left(\frac{1}{2}+\frac{1}{m}\right)_j}{j!^3}\sum\limits_{n=0}^j \frac{(-j)_n^3(\frac{1}{2})_n(\frac{1}{2}-\frac{1}{m})_n(\frac{1}{2}+\frac{1}{m})_n}{(\frac{1}{2}-j)_n(\frac{1}{2}+\frac{1}{m}-j)_n(\frac{1}{2}-\frac{1}{m}-j)_n n!^3}.
\end{equation}
Thus for each $m \in \lbrace 3,4,6 \rbrace$, $Z_m$ meets the conditions for Corollary \ref{cor:1.4} with $\varepsilon_0=\varepsilon_1=0$.  

We use Theorem \ref{thm:Sturm} to recognize $Z_m$ in terms of common modular forms.  In each case the weight is $k=4$, so it suffices to check the Fourier coefficients up to the $q^1$ term.  We obtain that
\begin{align}
Z_3(\tau) &= E_4(\tau), \label{eq:Z3} \\
Z_4(\tau) &= E_{2,2}(\tau)^2, \label{eq:Z4}\\
Z_6(\tau) &= \frac14 E_{2,3}(\tau)^2.\label{eq:Z6}
\end{align}
Write $f':=q\frac{df}{dq}= \frac{1}{2\pi i}\frac{df}{d\tau}$.  We next compute each $U_m$ so that 
\begin{equation}
t_{2,m}' = t_{2,m}Z_mU_m.
\end{equation}
When $m=3$, we have $t_{2,3}=1728/j$.  Thus,
\[
t_{2,3}' = \left(\frac{-j'}{j}\right)t_{2,3},
\]
and using Theorem \ref{thm:Sturm} with $k=6$ it is easy to check that
\[
-j'E_4 = jE_6,
\]
so from \eqref{eq:Z3} we obtain that 
\begin{equation}\label{eq:U3}
U_3(\tau)= \frac{E_6(\tau)}{E_4(\tau)^2}. 
\end{equation}
For $m=4,6$ we first observe that 
\[
t_{2,m} = \frac{4\alpha fg}{(f+\alpha g)^2},
\]
where $f(\tau)=\eta(\tau)^{k_m}$, $g(\tau)=\eta(m\tau/2)^{k_m}$, $k_m=48/(m-2)$, and $\alpha=(m/2)^{12/(m-2)}$.  Using the fact that $\eta'=\frac{1}{24}\eta E_2$, differentiating and simplifying yields
\begin{align*}
t_{2,m}' &= \frac{4\alpha [-fg(f'+\alpha g' + (f^2g' + \alpha f' g^2)]}{(f+\alpha g)^3} \\
&= t_{2,m}\frac{(f-\alpha g)}{(f+\alpha g)} \left[\frac{g'}{g} - \frac{f'}{f}\right]\\
&= t_{2,m}\frac{(f-\alpha g)}{(f+\alpha g)}\left(\frac{-k_m}{24}\right) E_{2,\frac{m}{2}}\\
&= t_{2,m}Z_m \frac{(2-m)}{2E_{2,\frac{m}{2}}} \frac{(f-\alpha g)}{(f+\alpha g)}.
\end{align*}
Thus we have that
\begin{align} 
U_4(\tau) &= \frac{-1}{E_{2,2}(\tau)}\cdot \frac{(\Delta(\tau) - 64\Delta(2\tau))}{(\Delta(\tau) + 64\Delta(2\tau))}, \label{eq:U4} \\
U_6(\tau) &= \frac{-2}{E_{2,3}(\tau)}\cdot \frac{(\eta(\tau)^{12} - 27\eta(3\tau)^{12})}{(\eta(\tau)^{12} + 27\eta(3\tau)^{12})}. \label{eq:U6}
\end{align}

Note that in each case we can see by Theorem \ref{thm:Sturm} that \begin{equation}\label{eq: U_m^2}
    U_m^2(\tau) = \frac{1-t_{2,m}(\tau)}{Z_m(\tau)}.
\end{equation}
Moreover, we have that $1/U(\tau)$ is a meromorphic modular form of weight $2$ on $\SL_2(\ZZ)$ so 
\begin{equation}\label{eq: U_m modular}
U_m(-1/s\tau) = (\sqrt{s}\tau)^{-2} U_m(\tau).
\end{equation}
For each $m\in \{3,4,6\}$, let $X = t_{2,m}$, $Y=X(N\tau)$, and $U = U_m$. In order to compute examples via Corollary \ref{cor:1.4} we need to know the derivative $\frac{dM_N}{dX}$.  By Lemma \ref{lem:M_N} we have that 
\begin{equation}\label{eq: M_N}
M_N(\tau)=N\frac{dX}{dY}\frac{Y}{X}(\tau) \frac{U(N\tau)}{U(\tau)},
\end{equation}
so 
\begin{equation}\label{eq:M_N^2}
M_N(\tau)^2=N^2\left(\frac{dX}{dY}\frac{Y}{X}(\tau)\right)^2 \frac{U(N\tau)^2}{U(\tau)^2}.
\end{equation}
From \eqref{eq: U_m^2} we have 
\begin{equation}
    \frac{U(N\tau)^2}{U(\tau)^2} = \left( \frac{1-Y(\tau)}{Z_m(N\tau)} \cdot \frac{Z_m(\tau)}{1-X(\tau)}\right) = \frac{1-Y}{1-X}M_N(\tau).
\end{equation}
Therefore 
\begin{equation*}
     M_N(\tau)^2=N^2\left(\frac{dX}{dY}\frac{Y}{X}\right)^2 \frac{1-Y}{1-X}M_N(\tau),
\end{equation*}
and so 
\begin{equation}
    M_N =N^2\left(\frac{dX}{dY}\frac{Y}{X}\right)^2 \frac{1-Y}{1-X}.
\end{equation}
Differentiating with respect to $X$ we get 
\begin{multline}\label{eq:dM_N}
    \frac{dM_N}{dX}= 2N^2\left(\frac{dX}{dY}\frac{Y}{X}\right)\left(\frac{1-Y}{1-X}\right)\left( \frac{d\left(\frac{dX}{dY}\right)}{dX} \frac{Y}{X}+ \frac{dX}{dY}\frac{X\left(\frac{dY}{dX}\right)-Y}{X^2}\right)+\\ 
    N^2\left(\frac{dX}{dY}\frac{Y}{X}\right)^2\frac{(1-X)\left(-\frac{dY}{dX}\right)+(1-Y)}{(1-X)^2}.
\end{multline}
For each $m \in \lbrace 3,4,6 \rbrace$ we construct specific examples of Corollary 1.4 for some choices of $N$. For each example we need to compute the special values of $X,Y$, and $U_m$, and find an explicit polynomial relationship between $X(\tau):=t_{2,m}(\tau)$ and $Y(\tau):= X(N\tau)$ in order to compute the special value of $\frac{dM_N}{dX}$. 

\subsection{\texorpdfstring{$\PSL_2(\ZZ)\cong(2,3,\infty) $}{}}
For Examples \ref{ex:m=3 N=2} and \ref{ex: m=3 N=3} we consider $m = 3$, $s = 1$, and $\gamma = \smat{0}{-1}{1}{0}$.

As in Tables \ref{tab:Hauptmoduln} and \ref{tab:gens}, we define \[X(\tau):= t_{2,3}(\tau) = \frac{1728}{j(\tau)},\]
and for the corresponding domain $D$ we choose the standard fundamental domain 
$$
 \{\tau\in \hh:\, |\mbox{Re}(\tau)|\leq 1/2,\, |\tau|>1\}.
$$
Let \[Z(\tau):= Z_3(\tau) = \sum\limits_{j=0}^\infty A_j X^j(\tau)= E_{4}(\tau),\]
where by \eqref{eq: A_j,m general},
 \begin{equation}\label{eqn: Aj m=3}A_j = \frac{(\frac{1}{6})_j (\frac{5}{6})_j (\frac{1}{2})_j}{j!^3}\cdot \sum\limits_{n=0}^j \frac{(-j)_n(\frac{1}{6})_n(\frac{5}{6})_n(\frac{1}{2})_n}{(\frac{5}{6}-j)_n (\frac{1}{6}-j)_n(\frac{1}{2}-j)_n n!^3}.
\end{equation}
Furthermore, define $Y(\tau):= X(N\tau)$ and $U(\tau) := U_3(\tau)$.

\begin{example}\label{ex:m=3 N=2}
Let $N=2$ and $\tau_0 = \frac{i}{\sqrt{2}}$. We need to determine $X(i/\sqrt{2})$ and $Y(i/\sqrt{2}) = X(i\sqrt{2})$. From the values of $j(i/\sqrt{2})$ and $j(i\sqrt{2})$  in Table \ref{tab:jvalues} we have
\[
X(i/\sqrt{2}) = Y(i/\sqrt{2}) = \frac{27}{125}.
\]
Using the polynomial relationship $\Phi_2(X,Y)=0$ from Lemma \ref{lem: modpoly m=3} and \eqref{eq:dM_N} we find $\frac{dM_2}{dX}\left\vert_{X(\tau_0)} = -\frac{500}{63} \right.$.

Letting $\tau=\sqrt{2}i$ in \eqref{eq: U_m modular} gives that
\begin{equation}\label{eq: U_3 modular}
U(i\sqrt{2}) =\frac{1}{2} U(i/\sqrt{2}), 
\end{equation}
Using \eqref{eq: U_m^2}, we can determine $U(i\sqrt{2})$ using the value of $X(i\sqrt{2})$ from above and $Z_3(i\sqrt{2}) = E_4(i\sqrt{2})$ from Table \ref{tab:Ekvalues} to determine that
\begin{equation*}
U(\sqrt{2}i) = \pi^3 \frac{2^5\cdot 7}{5^2}\G\left(\frac18\right)^{-2}\G\left(\frac38\right)^{-2},
\end{equation*}
and thus from \eqref{eq: U_m modular}
\begin{equation*}
U(i/\sqrt{2}) = -\pi^3 \frac{2^4\cdot 7}{5^2}\G\left(\frac18\right)^{-2}\G\left(\frac38\right)^{-2}.
\end{equation*}
Since $\tau_0 = \frac{i}{\sqrt{2}}\not\in D$ but $\gamma\tau_0 \in D$, Corollary ~\ref{cor:1.4} implies that 
\[
\frac{2\sqrt{2}}{\pi}
=\sum_{j=0}^\infty \left( a_2+b_2 j\right)A_jX^j(\sqrt{2}i),
\]
where 
\begin{align*}
a_N &= U(i/\sqrt{2})X(i/\sqrt{2})  \left.\frac{dM_N}{dX}\right\vert_{X=X(i/\sqrt{2})} = \frac{\pi^3\cdot 3\cdot 2^6}{5^2} \G\left(\frac18\right)^{-2}\G\left(\frac38\right)^{-2},\\
b_N &= 4U(\sqrt{2}i) = \frac{\pi^3\cdot 2^7\cdot 7}{5^2}\G\left(\frac18\right)^{-2}\G\left(\frac38\right)^{-2}. 
\end{align*}
This can be written as
\[ \frac{5^2}{\pi^4 }=2^{9/2}\Gamma\left(\frac{1}{8}\right)^{-2}\Gamma \left( \frac 38\right)^{-2} \sum_{j=0}^\infty \left(3+14j\right)A_j\left(\frac{27}{125}\right)^j,\]
where $A_j$ is as in \eqref{eqn: Aj m=3}.

\end{example}
\begin{example}\label{ex: m=3 N=3} Consider $N=3$ so $\tau_0 = \frac{i}{\sqrt{3}}$. Using the values $j(i/\sqrt{3})$ and $j(i\sqrt{3})$ from Table \ref{tab:jvalues} we get 
\[
X(i/\sqrt{3}) = Y(i/\sqrt{3}) = \frac{4}{125}.
\]
Using the polynomial relationship $\Phi_3(X,Y)=0$ from Lemma \ref{lem: modpoly m=3} and \eqref{eq:dM_N} we find
$\frac{dM_3}{dX}\left\vert_{X = X(i/\sqrt{3})}= -\frac{1125}{11}\right.$. Furthermore,
using \eqref{eq: U_m^2}, the value $X(i\sqrt{3})$ above, and the value of $E_4(i\sqrt{3})$ from Table \ref{tab:Ekvalues} we have
\[U(\sqrt3 i)=\sqrt{\frac{1-\frac{4}{125}}{E_4(\sqrt{3}i)}}=\frac{2^{14/3}\cdot11\cdot\pi^4 }{3\cdot 5^2\cdot\G\left(\frac 13\right)^{6}}.\]
Thus from \eqref{eq: U_m modular} we obtain
\[U(i/\sqrt{3})=-\frac{1}{3}U(\sqrt3 i)=-\frac{2^{14/3}\cdot11\cdot \pi^4 }{3^2\cdot 5^2\cdot \G\left(\frac 13\right)^{6}}.\]
Since $\tau = \frac{i}{\sqrt{3}} \not\in D$, and $\gamma\tau = \sqrt{3}i \in D$, Corollary ~\ref{cor:1.4} implies that
\[\frac{2\sqrt{3}}{\pi} = \sum\limits_{j=0}^\infty (a_3+b_3j)A_j X^j(e^{-(2/\sqrt{3})\pi}),\]
where
\begin{align*}
    a_3 &= \frac{2^{20/3}\cdot \pi^4}{5^2}\G\left(\frac{1}{3}\right)^{-6},\\
     b_3 &= \frac{2^{17/3}\cdot11\cdot  \pi^4}{5^2}\G\left(\frac{1}{3}\right)^{-6}.\\
\end{align*}
We can write this as 
\begin{equation}
\frac{\sqrt{3}\cdot 5^2}{\pi^5} = 2^{14/3}\G \left(\frac{1}{3}\right)^{-6} \sum\limits_{j=0}^\infty (11j+2)A_j \left(\frac{4}{125}\right)^j,
\end{equation}
where 
$A_j$ is as in \eqref{eqn: Aj m=3}.
 \end{example}

\subsection{\texorpdfstring{$ \G_0^+(2) \cong  (2,4,\infty)$}{}}
For example \ref{ex: m=4,N=3} we consider the case where $m = 4$. Here we have that $s=2$ so $\tau = \frac{1}{\sqrt{2}}\smat{0}{-1}{2}{0}$. 
 
As in Tables \ref{tab:Hauptmoduln} and \ref{tab:gens}, we define \[X(\tau):= t_{2,4} = \frac{256 \eta(\tau)^{24} \eta(2\tau)^{24}}{(\eta(\tau)^{24}+64\eta(2\tau)^{12})^2},\]
and for the corresponding domain $D$ we choose the intersection of $\{\tau\in \hh:\, |X(\tau)|<1\}$ and the following fundamental domain $FD$ of $\G_0^+(2)$
$$
 FD=\{\tau\in \hh:\, |\mbox{Re}(\tau)|\leq 1/2,\, |\tau|>1/\sqrt 2 \}.
$$
Let \[Z(\tau):= Z_4(\tau) = \sum\limits_{j=0}^\infty A_j X^j(\tau)= E_{2,2}(\tau)^2,\]
where by \eqref{eq: A_j,m general},
\begin{equation}\label{eq: A_j m=4}
A_j =  \frac{\left(\frac{1}{2}\right)_j\left(\frac{1}{2}-\frac{1}{4}\right)_j\left(\frac{1}{2}+\frac{1}{4}\right)_j}{j!^3}\sum\limits_{n=0}^j \frac{(-j)_n^3(\frac{1}{2})_n(\frac{1}{4})_n(\frac{3}{4})_n}{(\frac{1}{2}-j)_n(\frac{3}{4}-j)_n(\frac{1}{4}-j)_n n!^3}.
\end{equation}
Furthermore, define $Y(\tau):= X(N\tau)$ and $U(\tau) := U_4(\tau)$.

\begin{example}\label{ex: m=4,N=3}
Let $N=3$ and $\tau_0 = i/\sqrt{6}$. Using the values of $\eta(i/\sqrt{6})$, $\eta(i\sqrt{2/3})$, $\eta(i\sqrt{3/2})$, and $\eta(i\sqrt{6})$  from Table \ref{tab:etavalues} we get that
\[X(i/\sqrt{6}) = Y(i/\sqrt{6}) = \frac{1}{9}.\]
Using the polynomial relationship $\Phi_3(X,Y)=0$ from Lemma \ref{lemma:modpoly m=4 N=3} and \eqref{eq:dM_N} we find $\frac{dM_3}{dX}\left\vert_{X= X(i/\sqrt{6})} = -\frac{81}{2}\right.$. Recall from \eqref{eq: U_m^2} that  
\[U(\tau)= \frac{\sqrt{1-X(\tau)}}{E_{2,2}(\tau)}.\] 
Using the value of $E_{2,2}(i/\sqrt{6})$ from Table \ref{tab:E2kvalues} we determine that
\begin{equation}
 U\left(\frac{i}{\sqrt{6}}\right) = -\frac{32\pi^3}{3\sqrt{3}}\left(\G\left(\frac{1}{24}\right)\G\left(\frac{5}{24}\right)\G\left(\frac{7}{24}\right)\G\left(\frac{11}{24}\right)\right)^{-1},
 \end{equation} 
 and furthermore from \eqref{eq: U_m modular}, 
 \begin{equation}
    U\left(i\sqrt{\frac{3}{2}}\right) = \frac{32\pi^3}{\sqrt{3}}\left(\G\left(\frac{1}{24}\right)\G\left(\frac{5}{24}\right)\G\left(\frac{7}{24}\right)\G\left(\frac{11}{24}\right)\right)^{-1}.
\end{equation}
Thus Corollary \ref{cor:1.4} yields that 
\begin{equation}
    \frac{2\sqrt{6}}{\pi} = \sum\limits_{j=0}^\infty(b_3j+a_3)A_j\left(\frac{1}{9}\right)^j,
\end{equation}
where 
\begin{align*}
    a_3 &=2^4\sqrt{3}\pi^3\left(\G\left(\frac{1}{24}\right)\G\left(\frac{5}{24}\right)\G\left(\frac{7}{24}\right)\G\left(\frac{11}{24}\right)\right)^{-1}, \\
    b_3 &=2^6\sqrt{3}\pi^3\left(\G\left(\frac{1}{24}\right)\G\left(\frac{5}{24}\right)\G\left(\frac{7}{24}\right)\G\left(\frac{11}{24}\right)\right)^{-1},
\end{align*}
and $A_j$ is as in \eqref{eq: A_j m=4}.

We can write this as 
\[\frac{1}{\pi^4} = 2^{5/2}\left(\G\left(\frac{1}{24}\right)\G\left(\frac{5}{24}\right)\G\left(\frac{7}{24}\right)\G\left(\frac{11}{24}\right)\right)^{-1} \sum\limits_{j=0}^\infty (4j+1)A_j \left(\frac{1}{9}\right)^j,\]
where $A_j$ is as in \eqref{eq: A_j m=4}.
\end{example}

\subsection{\texorpdfstring{$  \G_0^+(3)\cong(2,6,\infty)$}{}}
For Examples \ref{ex: m=6, N=2} and \ref{ex: m=6 N=5} we consider $m =6$, $s=3$, and $\gamma = \frac{1}{\sqrt{3}} \smat{0}{-1}{3}{0}$. As in Tables \ref{tab:Hauptmoduln} and \ref{tab:gens}, we define \[X(\tau):= t_{2,6}(\tau) = \frac{ 108\eta(\tau)^{12}\eta(3\tau)^{12}}{(\eta(\tau)^{12}+27\eta(3\tau)^{12})^2},\]
and for the corresponding domain $D$ we choose the intersection of $\{\tau\in \hh:\, |X(\tau)|<1\}$ and the following fundamental domain $FD$ of $\G_0^+(3)$
$$
 \{\tau\in \hh:\, |\mbox{Re}(\tau)|\leq 1/2,\, |\tau|>1/\sqrt 3 \}.
$$
Let \[Z(\tau):= Z_6(\tau) = \sum\limits_{j=0}^\infty A_j X^j(\tau)= \frac{1}{4}E_{2,3}(\tau)^2,\]
where 
\begin{equation} \label{eqn: A_j m=6}
    A_j = \frac{\left(\frac{1}{2}\right)_j\left(\frac{1}{3}\right)_j\left(\frac{2}{3}\right)_j}{j!^3}\sum\limits_{n=0}^j\frac{\left(-j\right)_n\left(\frac{1}{2}\right)_n\left(\frac{1}{3}\right)_n\left(\frac{2}{3}\right)_n}{\left(\frac{1}{2}-j\right)_n\left(\frac{2}{3}-j\right)_n\left(\frac{1}{3}-j\right)_n n!^3}.
\end{equation} 
Furthermore, define $Y(\tau):= X(N\tau)$ and $U(\tau) := U_6(\tau)$. 
Recall from \eqref{eq:U6} that 
\[ U(\tau) = \frac{-2}{E_{2,3}(\tau)}\cdot \frac{(\eta(\tau)^{12} - 27\eta(3\tau)^{12})}{(\eta(\tau)^{12} + 27\eta(3\tau)^{12})}.\]
Using \eqref{eqn: thetaE23} we obtain 
\begin{equation} \label{eqn: U_6 as an eta quotient}
    U(\tau)=\frac{(\eta(\tau)^{12}-27\eta(3\tau)^{12})\eta(\tau)^2}{(\eta(\tau)^{12}+27\eta(3\tau)^{12})(3\eta(3\tau)^3+\eta(\tau/3)^3)^2}.
\end{equation}

\begin{example}\label{ex: m=6, N=2} Let $N=2$ and $\tau = \frac{i}{\sqrt{6}}$. We use the values of $\eta\left(i/\sqrt{6}\right)$, $\eta\left(i\sqrt{6}\right)$, $\eta\left(i\sqrt{2/3}\right)$, $\eta\left(i\sqrt{3/2}\right)$, $\eta\left(i/3\sqrt{6}\right)$ from Table \ref{tab:etavalues} to compute $X(i/\sqrt{6}) = Y(i/\sqrt{6}) = \frac{1}{2}$. Using the polynomial relationship $\Phi_2(X,Y)=0$ from Lemma \ref{lem: modpoly m=6} and \eqref{eq:dM_N} we get that $\frac{dM_2}{dX}\left\vert_{X = X(i/\sqrt{6})}=-\frac{16}{3}\right.$.

Using \eqref{eqn: U_6 as an eta quotient} and the necessary $\eta-$values from from Table \ref{tab:etavalues} we obtain
\[U\left(\frac{i}{\sqrt{6}}\right)= -\frac{8\sqrt{2}\pi^3}{\sqrt{3}},\]
and from the modularity of $U$ \eqref{eq: U_m modular}, 
\[U\left(i\sqrt{\frac{2}{3}}\right)= \frac{16\cdot\sqrt{2}\pi^3}{\sqrt{3}}.\]
Since $2\tau = i\sqrt{\frac{2}{3}}\in D$, Corollary \ref{cor:1.4} implies that 
\begin{equation}
    \frac{2\sqrt{6}}{\pi} = \sum\limits_{j=0}^\infty(b_2j+a_2)A_j\left(\frac{1}{2}\right)^j,
\end{equation}
where 
\begin{align*}
    a_2 &= -\frac{8}{3}U\left(\frac{i}{\sqrt{6}}\right),\\
    b_2 &= 4U\left(i\sqrt{\frac{2}{3}}\right) = -8U\left(\frac{i}{\sqrt{6}}\right).
\end{align*}
We can write this as \begin{equation}
    \frac{9}{\pi^4}= 2^5\cdot\left( \G\left(\frac{1}{24}\right)\G\left(\frac{5}{24}\right)\G\left(\frac{7}{24}\right)\G\left(\frac{11}{24}\right)\right)^{-1}\sum\limits_{j=0}^\infty (3j+1)A_j \left(\frac{1}{2}\right)^j,
\end{equation}
where $A_j$ is as in \eqref{eqn: A_j m=6}.
\end{example}
\begin{example}\label{ex: m=6 N=5} Let $N=5$ and $\tau_0 = i/\sqrt{15}$. Using the values of $\eta(i/\sqrt{15})$, $\eta(i\sqrt{3/5})$, $\eta(i\sqrt{5/3})$ and $\eta(i\sqrt{10})$ from Table \ref{tab:etavalues} we find that \[X(i/\sqrt{10}) = Y(i/\sqrt{10}) = \frac{4}{125}.\] Using the polynomial relationship $\Phi_5(X,Y)=0$ from Lemma \ref{lem: modpoly m=6} and \eqref{eq:dM_N} we get $\frac{dM_5}{dX}\left\vert_{X= X(i/\sqrt{15})} = \frac{-12500}{33} \right.$. Furthermore, from \eqref{eqn: U_6 as an eta quotient} and \eqref{eq: U_m modular} we use the values of $\eta(i/\sqrt{15}),\eta(i\sqrt{3/5})$, and $\eta(i/3\sqrt{15})$ from Table \ref{tab:etavalues} to obtain 
\[U(i/\sqrt{15})= -\frac{352\pi^3}{25\sqrt{15}} \left(\G\left(\frac{1}{15}\right)\G\left(\frac{2}{15}\right)\G\left(\frac{4}{15}\right)\G\left(\frac{8}{15}\right)\right)^{-1},\] and \[U(i\sqrt{5/3}) = -5U(i/\sqrt{10}) = \frac{352\pi^3}{5\sqrt{15}} \left(\G\left(\frac{1}{15}\right)\G\left(\frac{2}{15}\right)\G\left(\frac{4}{15}\right)\G\left(\frac{8}{15}\right)\right)^{-1}.\]
By Corollary \ref{cor:1.4} we have that 
\[ \frac{2\sqrt{15}}{\pi} = \sum\limits_{j=0}^\infty (b_5j+a_5)A_j\left(\frac{4}{125}\right)^j, \]
where 
\begin{align*}
    b_5&=10U(i\sqrt{5/3})= \frac{10\cdot 352\pi^3}{5\sqrt{15}}\left(\G\left(\frac{1}{15}\right)\G\left(\frac{2}{15}\right)\G\left(\frac{4}{15}\right)\G\left(\frac{8}{15}\right)\right)^{-1},\\
    a_5 &=  \frac{400}{33}\cdot  \frac{10\cdot 352\pi^3}{5\sqrt{15}}\left(\G\left(\frac{1}{15}\right)\G\left(\frac{2}{15}\right)\G\left(\frac{4}{15}\right)\G\left(\frac{8}{15}\right)\right)^{-1}.
\end{align*}
We can write this as 
\begin{equation*}
    \frac{3^2\cdot 5}{\pi^4}= 2^5\left(\G\left(\frac{1}{15}\right)\G\left(\frac{2}{15}\right)\G\left(\frac{4}{15}\right)\G\left(\frac{8}{15}\right)\right)^{-1} \cdot  \sum\limits_{j=0}^\infty \left( 33j+8\right)A_j \left(\frac{4}{125}\right)^j,
\end{equation*}
where $A_j$ is as in \eqref{eqn: A_j m=6}.
\end{example}

\section{Appendix -- Modular Polynomials} 
\label{sec:modularpoly}
Throughout, let $\Gamma$ be a discrete subgroup of $\SL_2(\RR)$ commensurable with $\SL_2(\ZZ)$. Further assume $\Gamma$ is of genus zero and contains a principal congruence subgroup $\Gamma(N)$. Let $t$ be a Hauptmodul of $\Gamma$ and $N$ be the smallest positive integer such that $\Gamma$ contains $\Gamma(N)$. For a positive integer $m$ coprime to $N$, the modular polynomial of level $m$ is defined to be the polynomial $\Phi_m(x,y)$ of minimal degree (up to scalar) such that for $\alpha \in \GL_2(\QQ)$ with $\det \alpha =m$,
$$
  \Phi_m(x,t(\tau))=\prod_{\gamma \in \Gamma\backslash \Gamma\alpha \Gamma }(x-t(\gamma\tau)). 
$$
Below we state as lemmas each of the modular polynomials that we use in this article.  In addition to using the Fourier expansion of $t$ to find modular polynomials computationally, we prove some of the lemmas to illustrate how we can obtain modular polynomials using known modular polynomials computed using the method of Br\"oker, Lauter, and Sutherland \cite{BLS} and the covering maps between modular curves.

\subsection{Modular polynomials for \texorpdfstring{$t_{2,3}$, $t_{2}$, $t_{3}$}{} and \texorpdfstring{$t_{\infty}$}{}}
We first recall results for groups which are subgroups of $\PSL_2(\ZZ)$.
\begin{lemma}
\label{lem: modpoly m=3}
For $t_{2,3}(\tau)=1728/j(\tau)$, the level-$2$ and level-$3$ modular polynomials are, respectively,   
\begin{align*}
\Phi_2(X,Y)=&1728(X^3+Y^3) -162000(X^3Y+XY^3) +2571264(X^2Y+XY^2) -2985984XY \\
&+5062500(X^3Y^2+X^2Y^3) +40773375X^2Y^2 -52734375X^3Y^3,\\
\Phi_3(X,Y)=& (1728)^2(X^4+Y^4) -(1728)^4XY+(1728)^3(2232)(X^2Y+XY^2)\\  
&-(1728)^2(1069956)(X^3Y+XY^3)+(1728)(36864000)(X^4Y+XY^4)\\
&+(1728)^2(2587918086)X^2Y^2  +(1728)(8900222976000)(X^2Y^3+X^3Y^2)\\
&+ 452984832000000(X^2Y^4+X^4Y^2) -770845966336000000X^3Y^3 \\
&+ 1073741824000000000(X^4Y^3+X^3Y^4).
\end{align*}
\end{lemma}
\begin{proof}
Let $x=j(\tau)$ and $y=j(2\tau)$. From Sutherland \cite{Sutherland} we have that $x,y$ satisfy
\begin{multline} 
\label{level2jpoly}
0=(x^3+y^3)-162000(x^2+y^2)+ 1488(x^2y+xy^2) -x^2y^2 \\ +8748000000(x+y) +40773375xy -157464000000000.
\end{multline}
Since $X(\tau)=1728/x$, multiplying through by $(1728)^4/x^3y^3$ gives us the first result.

Similarly, for $x=j(\tau)$ and $y=j(3\tau)$, 
from Sutherland \cite{Sutherland} we have that $x,y$ satisfy
\begin{multline*}
0 = (x^4 + y^4) - x^3y^3 +2232(x^3y^2 + x^2y^3) -1069956(x^3y + xy^3) \\
+36864000(x^3 +y^3) +2587918086x^2y^2 \\
+8900222976000(x^2y + xy^2) + 452984832000000(x^2 + y^2) \\
-770845966336000000xy +1855425871872000000000(x + y).
\end{multline*}
Multiplying through by $(1728)^6/x^4y^4$ and simplifying, gives us the second result. 
\end{proof}

\begin{lemma}
\label{modulart2N3}
For $t_2(\tau)=-64\frac{\eta(2\tau)^{24}}{\eta(\tau)^{24}}$, the level-$3$ and level-$5$ modular polynomials are, respectively,
	\begin{align*}
	    \Phi_3(X,Y)=&
X^4+Y^4-4096 X^3Y^3-900 \left(X^3 Y+XY^3\right)+28422 X^2 Y^2\\
&+4608 \left(X^3 Y^2+X^2
   Y^3+X^2 Y+XY^2\right)-4096 XY,\\
      \Phi_5(X,Y)=&  X^6 +Y^6 - 16777216(XY+X^5Y^5) + 31457280(X^2Y+XY^2 +X^4Y^5+X^5Y^4) \\
      &-
    17940480(X^3Y+XY^3+X^3Y^5+X^5Y^3) + 3143680(X^4Y +X^2Y^5+X^5Y^2+XY^4) \\
    &- 90630(X^5Y +XY^5) +
    3709829120(X^2Y^2+X^4Y^4)+ 746465295(X^4Y^2+X^2Y^4)  \\
    &+ 6259476480(X^3Y^2+X^2Y^3+X^4Y^3+X^3Y^4)  - 33983400980X^3Y^3.  
	\end{align*}
\end{lemma}
\begin{proof}
Recall from Sutherland \cite{Sutherland} that the level-$3$ modular polynomial for the elliptic $j$-function is 
\begin{align*}
    \Phi_3(x,y)=&x^4+y^4-x^3y^3+2232(x^3y^2+x^2y^3)-1069956(x^3y+xy^3)+2587918086x^2y^2\\
    &+36864000(x^3+y^3)+8900222976000(x^2y+xy^2)+452984832000000(x^2+y^2)\\
    &-770845966336000000xy+1855425871872000000000(x+y). 
\end{align*}
Moreover the relation between $j$ and $t_2$ is $j=\frac{64(4t_2-1)^3}{t_2}$. Hence, the functions $s:=t_2(\tau)$ and $t:=t_2(3\tau)$ satisfy the equation 
$$
  \Phi_3\left(\frac{64(4s-1)^3}{s},\frac{64(4t-1)^3}{t}\right)=0. 
$$
Together with the Fourier expansions of $s$ and $t$, we obtain 
$$
-4096s^3t^3+4608(s^3t^2+s^2t^3)+s^4+t^4-900(s^3t+st^3)+28422s^2t^2+4608(s^2t+st^2)-4096st =0, 
$$
which gives the modular polynomial of level-$3$ for $t_2$. The proof for level-$5$ follows similarly.
\end{proof}
\begin{lemma}\label{modulart3N2}
For $t_3(\tau):=-27 \frac{\eta(3\tau)^{12}}{\eta(\tau)^{12}}$, the  level-$2$ modular polynomial is 
	\begin{align*}
	    \Phi_2(X,Y)=& X^3+Y^3 + 27X^2Y^2 - 24(X^2Y+XY^2) + 27XY.
	\end{align*}
\end{lemma}
\begin{proof}
The proof follows similarly to that of Lemmas~\ref{lem: modpoly m=3} and \ref{modulart2N3}, using the relation 
$$
  j=-27\frac{(t_3-1)(9t_3-1)^3}{t_3}. 
$$
\end{proof}
\begin{lemma}\label{modularpolytoo}
For $t_\infty(\tau):=-16\eta(\tau)^8\eta(4\tau)^{16}/\eta(2\tau)^{24}$, the functions $t_\infty(\tau)$ and $t_2(2\tau)$ satisfy the equation 
\begin{align*}
	     0=& X^2Y^2 - 2X^2Y + X^2 + 16XY - 16Y.
	\end{align*}
\end{lemma}


\medskip

\subsection{Modular polynomials for \texorpdfstring{$t_{2,4}$}{} and \texorpdfstring{$t_{2,6}$}{}}

\begin{lemma}\label{lemma:modpoly m=4 N=3}
For $t_{2,4}(\tau)=\frac{256\eta(\tau)^{24}\eta(2\tau)^{24}}{(\eta(\tau)^{24}+64\eta(2\tau)^{24})^2}$, the level-$3$ and level-$5$ modular polynomials are, respectively,
\begin{align*}  
\Phi_3(X,Y)=&X^4+Y^4+5308416X^4Y^4+442368(X^4Y^3+X^3Y^4)+13824(X^4Y^2+X^2Y^4)\\
&+192(X^4Y+XY^4)-14015488X^3Y^3+2058048(X^3Y^2+X^2Y^3)\\
&-19332(X^3Y+XY^3)+3622662X^2Y^2+79872(X^2Y+XY^2)-65536XY,\\
\Phi_5(X,Y)=&X^6+Y^6+451377585192960000(X^6Y^4+X^4Y^6)+761203159669407744X^5Y^5\\
   &+69657034752000(X^6Y^3+X^3Y^6)-609930927695462400(X^5Y^4+X^4Y^5)\\
   &+4031078400(X^6Y^2+X^2Y^6)-20244489582182400(X^5Y^3+X^3Y^5)\\
   &+154441688220057600X^4Y^4+103680(X^6Y+XY^6)+4666060857600(X^5Y^2+X^2Y^5)\\
   &+36839200367577600(X^4Y^3+X^3Y^4)-65094150(X^5Y+XY^5)\\
   &+98471158056975(X^4Y^2+X^2Y^4)-13453926179834900X^3Y^3\\
   &+1256857600(X^4Y+XY^4)+173582058905600(X^3Y^2+X^2Y^3)\\
   &-5655756800(X^3Y+XY^3)+24370885427200X^2Y^2\\
   &+8724152320(X^2Y+XY^2)-4294967296XY.
\end{align*}
\end{lemma}

\begin{proof}
In this case the group we are considering is $\Gamma_0(2)^{+2}:=\langle\Gamma_0(2), \omega_2\rangle$. However, we first consider $\Gamma_0(6)^{+2}:=\langle\Gamma_0(6), \omega_2\rangle$, which is an index-$4$ subgroup of $\Gamma_0(6)^{+2}$.  It is known that that $u=\left(\frac{\eta(6\tau)\eta(3\tau)}{\eta(\tau)\eta(2\tau)}\right)^4$ is a Hauptmodul on  
$\Gamma_0(6)^{+2}$ (see \cite{Chan-Verrill} for example), so $X$ can be written as a rational function of degree $4$ in $u$.  In particular, one can check that
$$
X=\frac{256u}{(1+27u)^4}.
$$
Next, we observe that since $\omega_3$ normalizes $\Gamma_0(6)^{+2}$, $u(\omega_3\tau)$ is also a hauptmodul and $u(\omega_3\tau) =\frac{au+b}{cu+d}(\tau)$ for some $\smat abcd \in \GL_2(\CC)$.  In particular, for $\omega_3=\frac 1{\sqrt 3}\smat 3{-2}6{-3}$, we have $\omega_3\tau=\frac{3\tau-2}{6\tau-3}$.  Thus using the transformation law for the $\eta$-function we can relate $u(\omega_3\tau)$ and $u(\tau)$, namely,
\[
u(\omega_3\tau) = \frac{\eta^4\left(6\frac{3\tau-2}{6\tau-3}\right)\eta^4\left(3\frac{3\tau-2}{6\tau-3}\right)}{\eta^4\left(\frac{3\tau-2}{6\tau-3}\right)\eta^4\left(2\frac{3\tau-2}{6\tau-3}\right)}
    =   \frac{\eta^4\left(\smat 3{-2}2{-1}\tau\right)\eta^4\left(\smat 3{-4}1{-1}2\tau\right)}{\eta^4\left(\smat 1{-2}2{-3}3\tau\right)\eta^4\left(\smat 1{-4}1{-3}6\tau\right)} = \frac1{81 u(\tau)}. 
\]
Hence, 
$$
Y(\tau):=X(\omega_3\tau)=\frac{256u^3}{(1+3u)^4},
$$
and the rational function determined by the relations $X=\frac{256u}{(1+27u)^4}$ and  $Y=\frac{256u^3}{(1+3u)^4}$ gives rise to the desired level-$3$ polynomial for $X$.

Next we obtain the polynomial for level-$5$. Consider the following Atkin-Lehner involutions for $\G_0(10)$,
$$
\omega_2=\frac 1{\sqrt 2}\smat2{-1}{10}{-4}, \quad \omega_5=\frac 1{\sqrt 5}\smat52{10}{5}.
$$
Similar to the previous case, we start with the relation between $X$ and the given Hauptmoul 
$u=\left(\frac{\eta(10\tau)\eta(5\tau)}{\eta(\tau)\eta(2\tau)}\right)^2$
for $\Gamma_0(10)^{+2}:=\langle\Gamma_0(10), \omega_2\rangle$. One can check that 
$$
  X=\frac{256u}{(1+25u)^4(25u^2+6u+1)},
$$
and 
$$
  u(\omega_5\tau)=\frac 1{25}
  \left(  \frac{\eta(\tau)\eta(2\tau)}{\eta(10\tau)\eta(5\tau)}
  \right)^2=\frac1{25}\frac1{u(\tau)}. 
$$
Hence, 
$$
   Y(\tau):=X(\omega_5\tau)=\frac{256u^5}{(u+1)^4(25u^2+6u+1)},
$$
from which we obtain the desired level-$5$ polynomial.
\end{proof}

\begin{lemma}
For $t_{2,6}(\tau) = \frac{ 108\eta(\tau)^{12}\eta(3\tau)^{12}}{(\eta(\tau)^{12}+27\eta(3\tau)^{12})^2}$, the level-$2$ and level-$5$ modular polynomials are, respectively,
\label{lem: modpoly m=6}
\begin{align*} 
 \Phi_2(X,Y)=& 4X^3Y^3-12(X^3Y^2+X^2Y^3)+12(X^3Y+XY^3)-381X^2Y^2-4(X^3+Y^3)\\
 &-336(X^2Y+XY^2)+432XY,\\
\Phi_5(X,Y)=&262144000000X^6Y^6+19660800000X^6Y^5+19660800000X^5Y^6+614400000X^6Y^4\\
&-2550877126656X^5Y^5+614400000X^4Y^6+10240000X^6Y^3+2094980505600X^5Y^4\\
&+2094980505600X^4Y^5+10240000X^3Y^6+96000X^6Y^2-128213414400X^5Y^3\\
&-4716435974400X^4Y^4-128213414400X^3Y^5+96000X^2Y^6+480X^6Y\\
&+1141065600X^5Y^2+3568236045600X^4Y^3+3568236045600X^3Y^4+1141065600X^2Y^5\\
&+480XY^6+X^6-1221150X^5Y+75265374975X^4Y^2-4489016056900X^3Y^3\\
&+75265374975X^2Y^4-1221150XY^5+Y^6+31422600X^4Y+309367560600X^3Y^2\\
&+309367560600X^2Y^3+31422600XY^4-160088400X^3Y+101058937200X^2Y^2\\
&-160088400XY^3+264539520X^2Y+264539520XY^2-136048896XY.
  \end{align*}
\end{lemma}

\section{Appendix -- Special Values} 
In this appendix, we list the special values used in this article in Tables \ref{tab:jvalues}, \ref{tab:Ekvalues}, \ref{tab:E2kvalues}, and \ref{tab:etavalues}. Most of the values are obtained by either applying the Chowla-Selberg formula (see for example \cite[Eq. (1)]{ChapmanHart}), or using known or obtainable values together with relations between modular functions.  We first give an example of each of these methods below. 
	
\begin{example}\label{CS_ex} Let $\Delta(\tau):=\eta(\tau)^{24}$, the normalized weight-$12$ Hecke eigenform  on $\PSL_2(\ZZ)$. Then   
		$$
		 \Delta(\sqrt{2}i) = \frac{1}{2^{33}\pi^{18}}\G\left( \frac18 \right)^{12} \!\!\! \G\left( \frac38 \right)^{12}.
		$$
\begin{proof}Using the Chowla-Selberg formula \cite[Eq. (1)]{ChapmanHart}, we compute $\Delta(\sqrt{2}i)$ in terms of Gamma functions.  In particular, $\QQ(\sqrt{-8})=\QQ(\sqrt{-2})$ has class number $1$, and the unique reduced binary quadratic form with discriminant $-8$ is $x^2+2y^2$. Moreover, $\pm 1$ are the only roots of unity in $\QQ(\sqrt{-2})$.  Thus the Chowla-Selberg formula gives that
		\[
		\Delta(\sqrt{2}i) = (16\pi)^{-6} \prod_{m=1}^8 \G\left( \frac{m}{8} \right)^{6\left(\frac{-8}{m}\right)},
		\]
		where in the powers of the Gamma values are Kronecker symbols.  Evaluating the Kronecker symbols we have that $\left(\frac{a}{m}\right)=0$ when $a,m$ are both even, and $\left( \frac{-8}{1} \right)=\left( \frac{-8}{3} \right) =1$, while $\left( \frac{-8}{5} \right)=\left( \frac{-8}{7} \right) = -1$.  Thus we obtain 
		\[
		\Delta(\sqrt{2}i) = \frac{1}{2^{24}\pi^6} \frac{\G\left( \frac18 \right)^6\G\left( \frac38 \right)^6}{\G\left( \frac58 \right)^6\G\left( \frac78 \right)^6}.
		\]
		By the $\G$-reflection formula, $\G(1/8)\G(7/8)=\frac{\pi}{\sin(\pi/8)}$ and $\G(3/8)\G(5/8)=\frac{\pi}{\sin(3\pi/8)}$.  Thus we have
		\[
		\Delta(\sqrt{2}i) = \frac{(\sin(\pi/8)\sin(3\pi/8))^6}{2^{24}\pi^{18}}\G\left( \frac18 \right)^{12} \!\!\!  \G\left( \frac38 \right)^{12}.
		\]
		Using half-angle trigonometric formulas we calculate that $\sin(\pi/8)=(1/2)\sqrt{2-\sqrt2}$ and $\sin(3\pi/8)=(1/2)\sqrt{2+\sqrt2}$ so that $(\sin(\pi/8)\sin(3\pi/8))^6=2^{-9}$, which gives the desired value.
\end{proof}
	\end{example}

\begin{example}
Let $X(\tau)$ be the Hauptmodul $-64\frac{\eta(2\tau)^{24}}{\eta(\tau)^{24}}$ for $\G_0(2)$. We have \begin{align*}
X(i/\sqrt{6}) & =-17-12\sqrt{2}, \\
X(i\sqrt{6}) & =6\sqrt2+\frac{71}{8}-\frac{21}4\sqrt 3-\frac{27}8\sqrt 6, \\
X(i\sqrt{3/2}) & =-17+12\sqrt2. 
\end{align*}
\end{example}	
\begin{proof} 
We first establish the $j$-value
\begin{equation}\label{jval}
j(i\sqrt{6})= j(i/\sqrt{6})= 1728(1399 + 988\sqrt2).
\end{equation}
Since $i\sqrt6$ is a CM point and $\QQ(i\sqrt6)$ a CM field of discriminant $-24$ and class number $2$, Class Field Theory gives that $\mathbb{Q}(j(i\sqrt6))=\mathbb{Q}(\sqrt2)$ and $j(i\sqrt6)\in \mathbb{Z}[\sqrt2]$.  Moreover, since the lattices $\ZZ + i\sqrt6\ZZ$ and $2\ZZ + i\sqrt6\ZZ$ are inequivalent as ideals of $\ZZ[i\sqrt6]$, the $j$-values $j(i\sqrt{6})$ and $j(i\sqrt{6}/2)$ are Galois conjugates, so there exist $a,b\in \ZZ$ such that $j(i\sqrt{6}) = a+b\sqrt2$ and $j(i\sqrt{6}/2) = a-b\sqrt2$.  These are the $j$-invariants of elliptic curves with CM discriminant $-24$ over a quadratic field, and the L-functions and Modular Forms Database \cite{lmfdb} shows that the only $j$-values of such curves are $1728(1399 + 988\sqrt2)$ and $1728(1399 - 988\sqrt2)$.  As the $j$-values are finite at these points, the Fourier expansion of $j$ allows the use of numerical approximation to determine which value is which.  (See \cite{Cox, Silverman-adv} for example.)

The relation between $j$ and $X$ is 
\begin{equation}\label{jrelation}
j=\frac{64(4X-1)^3}{X},
\end{equation}
and solving the equation $\frac{64(4X-1)^3}{X}=1728(1399 + 988\sqrt2)$, yields the three solutions
$$
6\sqrt 2+\frac{71}{8}-\frac{21}4\sqrt3-\frac{27}8\sqrt 6, \quad 6\sqrt 2+\frac{71}{8}+\frac{21}4\sqrt3+\frac{27}8\sqrt6, \quad  -17-12\sqrt{2}. 
$$ 
To nail down which of the three values above are $X(i/\sqrt{6})$ and $X(i\sqrt{6})$, one can plug $i/\sqrt 6$ and $i\sqrt 6$ respectively into the Fourier expansion of $X$ to approximate and recognize the desired values.  
		
To find the third value $X(i\sqrt{3/2})=X(i\sqrt{6}/2)$, we can use either of the following two methods.  First, recall from \eqref{level2jpoly} that level-$2$ polynomial satisfied by $x=j(\tau)$ and $y=j(2\tau)$ is
\begin{multline*}
\Phi_2(x,y)= x^3+y^3-x^2y^2+1488 (xy^2+x^2y)-162000(x^2+y^2)\\
+40773375xy+8748000000(x+y)-157464000000000.
\end{multline*}
Thus using \eqref{jval} to solve the equation 
$$
\Phi_2(x, j(i\sqrt{6}))= 0,
$$
yields the possible values of $j(i\sqrt{6}/2)$, so we can use the Fourier expansion of $j(\tau)$ as above to approximate $j(i\sqrt{6}/2)$ and determine which is the correct value.  Then using \eqref{jrelation} we can similarly deduce the desired value of $X(i\sqrt{3/2})$.

Alternatively, using the transformation law for the $\eta$-function we note that
$$
X\left(\frac{-1}{2\tau}\right)=-64 \frac{\eta\(\frac{-1}{\tau}\)^{24}}{\eta\(\frac{-1}{2\tau}\)^{24}}=\frac 1{X(\tau)}.
$$
Therefore, letting $\tau=i/\sqrt{6}$ gives that 
\begin{equation}
\label{t2value}
X(i\sqrt{3/2})=1/(-17-12\sqrt{2})=-17+12\sqrt2. 
\end{equation}
\end{proof}

\subsection{Tables of special values}

In Tables \ref{tab:jvalues} and \ref{tab:etavalues} we indicate references for known or determined values.  We also indicate when a value is obtained directly or from previous values in the table using one or more of the following methods labeled A-K below.  
\begin{itemize}
\item[A]: Use the transformation formula $\eta(-1/\tau)=\sqrt{-i\tau}\eta(\tau)$ together with the value for $\eta(-1/\tau_0)$ to find the value of $\eta(\tau_0)$.
\item[B]: Use the relationship between $t_2(\tau)=-64/j_{2B}(\tau)$ and $j(\tau)$ ($t_2$ is given in terms of $\eta(\tau)$ and $\eta(2\tau)$) along with the values of $j(\tau_0), \eta(\tau_0)$ to find the value of $\eta(2\tau_0)$.
\item[C]: Use the action of the Atkin-Lehner involution $W_3$ on $t_{2,6}(\tau)$ sending $i\sqrt{6}$ to $i/3\sqrt{6}$ along with the value of $\eta(i\sqrt{6})$.
\item[D]: Use the modular polynomial $\phi_3=x^4 + 36x^3 + 270x^2 - xj + 756x + 729$ relating $x=-27t_3(\tau)$ to $j(\tau)$ along with the values of $j(\tau_0)$ and $\eta(\tau_0)$ to find the value of $\eta(3\tau_0)$.
\item[E]: Use the modular polynomial relating $j(\tau)$ and $j(2\tau)$.
\item[F]: Apply the action of the matrix $S=\smat{0}{-1}{1}{0}$.
\item[G]: Apply the action of the matrix $T=\smat{1}{1}{0}{1}$.
\item[H]: Use the modular polynomial relating $j(\tau)$ and $j(5\tau)$.
\item[I]: Use the transformation formula $\eta(\gamma\tau)^{24}=(c\tau+d)^{12}\eta(\tau)^{24}$ for $\gamma \in SL_2(\mathbb{Z})$.
\item[J]: Use the Chowla-Selberg formula, as in Example \ref{CS_ex}.
\item[K]: Use Class Field Theory and the L-functions and Modular Forms Database, as in \eqref{jval}. 
\end{itemize}

In Table \ref{tab:etavalues} we also use the following definitions to preserve space
\begin{align*}
a &=(71639575 + 32038171\sqrt{5} + 77 \sqrt{2838511914270 + 1269421119050 \sqrt{5}}),\\
b &=(6 \cdot 10^{2/3} (9125 + 4081\sqrt{5})),\\
c &=(-647+288 \sqrt{5}+9\sqrt{2(5145-2300\sqrt{5})}).
\end{align*}

To compute the special values in Tables \ref{tab:Ekvalues} we utilize the relationship between $E_4$, $E_6$, and $\Delta$ and in \ref{tab:E2kvalues} we utilize the relationship $E_{2,2}(\tau) =-(\theta_3(2\tau)^4+\theta_2(2\tau)^4).$

\FloatBarrier
\begin{table}[h]
    \begin{tabular}{llll}
    \end{tabular}
\end{table}
\FloatBarrier

\begin{table}[h!]
\renewcommand{\arraystretch}{1.6}
    \caption{$j$-values}
    \label{tab:jvalues}
\begin{center}
\begin{tabular}{ |c|c|c| } 
\hline
$\tau$ & $j(\tau)$ & Method \\
\hline
$i\sqrt{2}$ & $8000$ & \cite{Zagier} \\
\hline
$i/\sqrt{2}$ & $8000$ & F \\
\hline
$(1+i\sqrt{3})/2$ & $0$ & \cite{Zagier}\\
\hline
$i\sqrt{3}$ & $54000$ & E, G\\
\hline
$i/\sqrt{3}$ & $54000$ & F\\
\hline
$2i\sqrt{3}$ & $40500(35010+20213\sqrt3)$ & E\\
\hline
$i\sqrt{3}/2$ & $40500(35010-20213\sqrt3)$ & E\\
\hline
$i\sqrt{6}$ & $1728(1399 + 988\sqrt2)$ & K \\
\hline
$i/\sqrt{6}$ & $1728(1399 + 988\sqrt2)$ &F\\
\hline
$i\sqrt{6}/2$ & $1728(1399 - 988\sqrt2)$ & K\\ 
\hline
$i\sqrt{2/3}$ & $1728(1399 - 988\sqrt2)$ & E\\
\hline
$(1+i\sqrt{6})/2$ & 
$216(27014055899 + 19101822064\sqrt2 - 15596572446\sqrt3 - 11028442113\sqrt6)$  & E, G\\
\hline
$i\sqrt{10}$ & $8640(24635+11016\sqrt5)$ & K \\
\hline
$i/\sqrt{10}$ & $8640(24635+11016\sqrt5)$ &F \\
\hline
 $i\sqrt{5/2}$ & $8640(24635-11016 \sqrt{5})$ & K\\
 \hline
 $i\sqrt{2/5}$ & $8640(24635-11016\sqrt{5})$ & F\\
 \hline
$i\sqrt{5}$ & $320(1975+884\sqrt5)$ & K\\
\hline
$i/\sqrt{5}$ & $320(1975+884\sqrt5)$ & F\\
\hline
$(1+i\sqrt{5})/2$ & $320(1975-884\sqrt5)$ & K\\
\hline
$(1+i/\sqrt{5})/2$ & $320(1975-884\sqrt5)$ & E, G\\
\hline
$(1+i\sqrt{15})/2$  & $-\frac{135}{2}(1415 + 637\sqrt{5})$ & \cite{Zagier} \\
\hline
$i\sqrt{15}$ & $\frac{135}{2} (274207975+122629507\sqrt{5})$ & E, G \\
\hline
$i/\sqrt{15}$  & $\frac{135}{2} (274207975+122629507\sqrt{5})$ & F \\
\hline
\end{tabular}
\end{center}
\end{table}

\FloatBarrier
\begin{table}[h]
    \begin{tabular}{llll}
    \end{tabular}
\end{table}
\FloatBarrier

\begin{table}[h!]
\renewcommand{\arraystretch}{1.6}
    \caption{$\eta$-values}
    \label{tab:etavalues}
\begin{center}
\begin{tabular}{ |c|c|c| } 
\hline
$\tau$ & $\eta(\tau)$ & Method \\
\hline
$i\sqrt3 $ &$ \frac{3^{1/8}}{2^{4/3}\pi} \G\left( 1/3\right)^{3/2}$ & \cite{LLT}\\
\hline
$  (-1+i\sqrt3)/2$ & $\frac{3^{1/8}}{e^{\pi i/24}2\pi}\G(1/3)^{3/2}$ & \cite{LLT} \\
\hline
$i\sqrt{3/2}$ & $\frac{1}{2^{5/4}3^{1/4} \pi^{3/4}}(\sqrt{2}+1)^{1/12}\left(\Gamma(\frac1{24})\Gamma(\frac5{24})\Gamma(\frac7{24})\Gamma(\frac{11}{24})\right)^{1/4}$ &  \eqref{t2value}, J \\ 
\hline
$i\sqrt{2/3}$ &$\frac1{2^{3/2}\pi^{3/4}}(\sqrt{2}+1)^{1/12}\left(\Gamma(\frac1{24})\Gamma(\frac5{24})\Gamma(\frac7{24})\Gamma(\frac{11}{24})\right)^{1/4}$ & A\\
\hline
$i\sqrt{6}$ & $\frac1{2^{5/4}6^{1/4}\pi^{3/4}}(\sqrt{2}-1)^{1/12}\left(\Gamma(\frac1{24})\Gamma(\frac5{24})\Gamma(\frac7{24})\Gamma(\frac{11}{24})\right)^{1/4}$ & \eqref{t2value}, J  \\ 
\hline
$i/\sqrt{6}$ & $\frac1{2^{5/4}\pi^{3/4}}(\sqrt{2}-1)^{1/12}\left(\Gamma(\frac1{24})\Gamma(\frac5{24})\Gamma(\frac7{24})\Gamma(\frac{11}{24})\right)^{1/4}$ & A \\
\hline
$2i\sqrt{2/3}$ & $(\frac3{32}\sqrt 2-\frac{71}{512}+\frac{21}{256}\sqrt 3-\frac{27}{512}\sqrt6)^{1/24} \eta\left(i\sqrt{2/3}\right) $ & E, B\\
\hline
$i/3\sqrt 6$ & $6^{1/4}\eta(i\sqrt 6) ((-12\sqrt 2+12)(12\sqrt2+17)^{2/3}-18\sqrt2 +36(12\sqrt2+17)^{1/3}-39)^{1/12}$ & C\\
\hline
$2i\sqrt{6}$ & $\left(\frac{3-2\sqrt{2}}{2359+1668\sqrt{2}+3\sqrt{1236594+874404\sqrt{2}}}\right)^{1/24} \!\!\!\!\!\! \cdot \frac{1}{ 6\cdot 2^{7/8}\cdot \pi^{3/4}} \left(\G\left(\frac{1}{24}\right)\G\left(\frac{5}{24}\right)\G\left(\frac{7}{24}\right)\G\left(\frac{11}{24}\right)\right)^{1/4}$ & B \\
\hline
$(1+i\sqrt{15})/2$ &$e^{\pi i/24} \frac{1}{2^{3/4}3^{1/4}5^{1/4}\pi^{3/4}}\left(\frac{1+\sqrt{5}}{2}\right)^{-1/12}{\left(\Gamma(\frac1{15})\Gamma(\frac2{15})\Gamma(\frac4{15})\Gamma(\frac{8}{15})\right)^{1/4}}$ & \cite{vanderPoortenWilliams}\\
\hline
$i\sqrt{15}$ & $\frac1{2^{5/3}3^{1/4}5^{1/4}\pi^{3/4}}({\sqrt{5}-1})^{5/12}{\left(\Gamma(\frac1{15})\Gamma(\frac2{15})\Gamma(\frac4{15})\Gamma(\frac{8}{15})\right)^{1/4}}$ & B, G\\
\hline
$i/\sqrt{15}$ & $\frac1{2^{5/3}\pi^{3/4}}({\sqrt{5}-1})^{5/12}{\left(\Gamma(\frac1{15})\Gamma(\frac2{15})\Gamma(\frac4{15})\Gamma(\frac{8}{15})\right)^{1/4}}$ & A\\
\hline
$i\sqrt{3/5} $ & $\frac{1}{2^{7/4} \pi^{3/4}}(\sqrt{5}-1)^{5/12} (123+55 \sqrt{5})^{1/12} \left(\frac{1}{3} \Gamma(\frac{1}{15}) \Gamma(\frac{2}{15}) \Gamma(\frac{4}{15}) \Gamma(\frac{8}{15})\right)^{1/4}$ & D\\
\hline
$i/3\sqrt{15}$ & $\frac{(3 (-377 - 165 \sqrt{5} - b/a^{1/3} + (10a)^{1/3}))^{1/12}}{2 (2 \pi)^{3/4}}(\sqrt{5}-1)^{5/12} \left(\Gamma(\frac1{15})\Gamma(\frac2{15})\Gamma(\frac4{15})\Gamma(\frac{8}{15})\right)^{1/4} $ & D, A\\
\hline
$i\sqrt{5/3}$ & $\frac{1}{15} 3^{1/4} 5^{3/4} \eta(i/3\sqrt{15})$ & A\\
\hline
$i\sqrt{5/2}$ & $\frac{(\sqrt{5}-1)^{-1/4}}{2^{3/2} 5^{1/4}\pi^{5/4}}\left(\Gamma(\frac{1}{40})\Gamma(\frac{7}{40})\Gamma(\frac{9}{40})\Gamma(\frac{11}{40})\Gamma(\frac{13}{40})\Gamma(\frac{19}{40})\Gamma(\frac{23}{40})\Gamma(\frac{37}{40})\right)^{1/4}$ & B, J \\
\hline
$i\sqrt{10}$ & $\frac{(\sqrt{5}-1)^{1/4}}{2^{9/4}5^{1/4}\pi^{5/4}}\left(\Gamma(\frac{1}{40})\Gamma(\frac{7}{40})\Gamma(\frac{9}{40})\Gamma(\frac{11}{40})\Gamma(\frac{13}{40})\Gamma(\frac{19}{40})\Gamma(\frac{23}{40})\Gamma(\frac{37}{40})\right)^{1/4}$ & B, J \\
\hline
$i/\sqrt{10}$ & $\frac{(\sqrt{5}-1)^{1/4}}{4\pi^{5/4}}\left(\Gamma(\frac{1}{40})\Gamma(\frac{7}{40})\Gamma(\frac{9}{40})\Gamma(\frac{11}{40})\Gamma(\frac{13}{40})\Gamma(\frac{19}{40})\Gamma(\frac{23}{40})\Gamma(\frac{37}{40})\right)^{1/4}$ & A\\
\hline
$i\sqrt{2/5}$ & $ \frac{(161+72\sqrt{5})^{1/24}}{22^{3/4}(1+\sqrt{5})^{1/4}\pi^{5/4}}\left(\Gamma(\frac{1}{40})\Gamma(\frac{7}{40})\Gamma(\frac{9}{40})\Gamma(\frac{11}{40})\Gamma(\frac{13}{40})\Gamma(\frac{19}{40})\Gamma(\frac{23}{40})\Gamma(\frac{37}{40})\right)^{1/4}$ & B\\
\hline
$2i\sqrt{2/5}$ & $\frac{\left((161+72\sqrt{5})c\right)^{1/24}}{{4\cdot2^{1/8}(1+\sqrt{5})^{1/4}\pi^{5/4}}}\left(\Gamma(\frac{1}{40})\Gamma(\frac{7}{40})\Gamma(\frac{9}{40})\Gamma(\frac{11}{40})\Gamma(\frac{13}{40})\Gamma(\frac{19}{40})\Gamma(\frac{23}{40})\Gamma(\frac{37}{40})\right)^{1/4}$ & E, B\\
\hline
\end{tabular}
\end{center} 
\end{table}

\FloatBarrier
\begin{table}[h]
    \begin{tabular}{llll}
    \end{tabular}
\end{table}
\FloatBarrier

\begin{table}[h!]
\renewcommand{\arraystretch}{1.8}
    \caption{$E_k$-values}
    \label{tab:Ekvalues}
\begin{center}
\begin{tabular}{ |c|c|c| } 
\hline
$\tau$ & $E_4(\tau)$ & $E_6(\tau)$ \\
\hline
$ i\sqrt{2}$ & $\frac{5}{2^{9}\pi^{6}}\G\left(\frac18\right)^4\G\left(\frac38\right)^4$ & $\frac{7}{2^{13}\pi^{12}}\G\left(\frac18\right)^6\G\left(\frac38\right)^6$ \\
\hline
$i\sqrt{3}$ & $\frac{3^25}{\sqrt[3]{2}\cdot 2^{9}\pi^{8}} \G\left(\frac 13\right)^{12}$ & $\frac{3^311}{ 2^{14}\pi^{12}} \G\left(\frac 13\right)^{18}$ \\
\hline
\end{tabular}
\end{center}
\end{table}

\FloatBarrier
\begin{table}[h]
    \begin{tabular}{llll}
    \end{tabular}
\end{table}
\FloatBarrier

\begin{table}[h!]
 \renewcommand{\arraystretch}{1.8}
     \caption{$E_{2,k}$-values}
    \label{tab:E2kvalues}
\begin{center}
\begin{tabular}{ |c|c|c| } 
\hline
$k$ & $\tau$ & $E_{2,k}(\tau)$ \\
\hline
2 & $i/\sqrt{6}$  &  \makecell{$- \frac{1+\sqrt{2}+(-1+\sqrt{2})^{2/3}(-7925+5695\sqrt{2}+4584\sqrt{3}-3282\sqrt{6})^{1/3}}{32(\sqrt{2}-1)^{2/3}(-71+48\sqrt{2}+42\sqrt{3}-27\sqrt{6})^{1/3}\pi^3}\Gamma\left(\frac{1}{24}\right)\Gamma\left(\frac{5}{24}\right)\Gamma\left(\frac{7}{24}\right)
\Gamma\left(\frac{11}{24}\right)$}\\
\hline
\end{tabular}
\end{center}
\end{table}

\FloatBarrier
\begin{table}[h]
    \begin{tabular}{llll}
    \end{tabular}
\end{table}
\FloatBarrier



\begin{thebibliography}{10}

\bibitem{AAR}
George~E Andrews, Richard Askey, Ranjan Roy, Ranjan Roy, and Richard Askey.
\newblock {\em Special functions}, volume~71.
\newblock Cambridge university press Cambridge, 1999.

\bibitem{Bauer}
Gustav Bauer.
\newblock Von den {C}oefficienten der {R}eihen von {K}ugelfunctionen einer
  {V}ariablen.
\newblock {\em J. Reine Angew. Math.}, 56:101--121, 1859.

\bibitem{BB}
Jonathan~M. Borwein and Peter~B. Borwein.
\newblock {\em Pi and the {AGM}}.
\newblock Canadian Mathematical Society Series of Monographs and Advanced
  Texts. John Wiley \& Sons Inc., New York, 1987.
\newblock A study in analytic number theory and computational complexity, A
  Wiley-Interscience Publication.

\bibitem{BLS}
Reinier Br{\"o}ker, Kristin Lauter, and Andrew Sutherland.
\newblock Modular polynomials via isogeny volcanoes.
\newblock {\em Mathematics of Computation}, 81(278):1201--1231, 2012.

\bibitem{BGHZ}
Jan~Hendrik Bruinier, Gerard van~der Geer, G\"{u}nter Harder, and Don Zagier.
\newblock {\em The 1-2-3 of modular forms}.
\newblock Universitext. Springer-Verlag, Berlin, 2008.
\newblock Lectures from the Summer School on Modular Forms and their
  Applications held in Nordfjordeid, June 2004, Edited by Kristian Ranestad.

\bibitem{CCL}
Heng~Huat Chan, Song~Heng Chan, and Zhiguo Liu.
\newblock Domb's numbers and ramanujan--sato type series for 1/$\pi$.
\newblock {\em Advances in Mathematics}, 186(2):396--410, 2004.

\bibitem{ChanCooper}
Heng~Huat Chan and Shaun Cooper.
\newblock Rational analogues of {R}amanujan's series for {$1/\pi$}.
\newblock {\em Math. Proc. Cambridge Philos. Soc.}, 153(2):361--383, 2012.

\bibitem{Chan-Verrill}
Heng~Huat Chan and Helena Verrill.
\newblock The {A}p\'{e}ry numbers, the {A}lmkvist-{Z}udilin numbers and new
  series for {$1/\pi$}.
\newblock {\em Math. Res. Lett.}, 16(3):405--420, 2009.

\bibitem{Chandrasekharan}
K.~Chandrasekharan.
\newblock {\em Elliptic functions}, volume 281 of {\em Grundlehren der
  mathematischen Wissenschaften [Fundamental Principles of Mathematical
  Sciences]}.
\newblock Springer-Verlag, Berlin, 1985.

\bibitem{ChapmanHart}
Robin Chapman and William Hart.
\newblock Evaluation of the dedekind eta function.
\newblock {\em Canadian Mathematical Bulletin}, 49(1):21--35, 2006.

\bibitem{ChoiKim}
SoYoung Choi and Chang~Heon Kim.
\newblock Valence formulas for certain arithmetic groups and their
  applications.
\newblock {\em Journal of Mathematical Analysis and Applications},
  420(1):447--463, 2014.

\bibitem{CC}
David~V. Chudnovsky and Gregory~V. Chudnovsky.
\newblock Approximations and complex multiplication according to {R}amanujan.
\newblock In {\em Ramanujan revisited ({U}rbana-{C}hampaign, {I}ll., 1987)},
  pages 375--472. Academic Press, Boston, MA, 1988.

\bibitem{ConwayNorton}
John~H Conway and Simon~P Norton.
\newblock Monstrous moonshine.
\newblock {\em Bulletin of the London Mathematical Society}, 11(3):308--339,
  1979.

\bibitem{Cox}
David~A. Cox.
\newblock {\em Primes of the form {$x^2 + ny^2$}}.
\newblock A Wiley-Interscience Publication. John Wiley \& Sons, Inc., New York,
  1989.
\newblock Fermat, class field theory and complex multiplication.

\bibitem{DiamondShurman}
Fred Diamond and Jerry Shurman.
\newblock {\em A first course in modular forms}, volume 228 of {\em Graduate
  Texts in Mathematics}.
\newblock Springer-Verlag, New York, 2005.

\bibitem{Kaiblinger}
Norbert Kaiblinger.
\newblock Product of two hypergeometric functions with power arguments.
\newblock {\em J. Math. Anal. Appl.}, 479(2):2236--2255, 2019.

\bibitem{LLT}
Wen-Ching~Winnie {Li}, Ling {Long}, and Fang-Ting {Tu}.
\newblock {Computing special L-values of certain modular forms with complex
  multiplication}.
\newblock {\em SIGMA 14 (2018), 090}, August 2018.

\bibitem{lmfdb}
The {LMFDB Collaboration}.
\newblock The {L}-functions and modular forms database.
\newblock \url{http://www.lmfdb.org}, 2021.
\newblock [Online; accessed 25 August 2021].

\bibitem{Maier}
Robert~S. Maier.
\newblock On rationally parametrized modular equations.
\newblock {\em J. Ramanujan Math. Soc.}, 24(1):1--73, 2009.

\bibitem{Miranda}
Rick Miranda.
\newblock {\em Algebraic curves and Riemann surfaces}, volume~5.
\newblock American Mathematical Soc., 1995.

\bibitem{Ramanujan}
Srinivasa Ramanujan.
\newblock Modular equations and approximations to {$\pi$} [{Q}uart. {J}.
  {M}ath. {\bf 45} (1914), 350--372].
\newblock In {\em Collected papers of {S}rinivasa {R}amanujan}, pages 23--39.
  AMS Chelsea Publ., Providence, RI, 2000.

\bibitem{Silverman-adv}
Joseph~H. Silverman.
\newblock {\em Advanced topics in the arithmetic of elliptic curves}, volume
  151 of {\em Graduate Texts in Mathematics}.
\newblock Springer-Verlag, New York, 1994.

\bibitem{Sturm}
Jacob Sturm.
\newblock On the congruence of modular forms.
\newblock In {\em Number theory}, pages 275--280. Springer, 1987.

\bibitem{Sutherland}
Drew Sutherland.
\newblock Modular polynomials.
\newblock \url{https://math.mit.edu/~drew/ClassicalModPolys.html}.

\bibitem{Takeuchi}
Kisao Takeuchi.
\newblock Arithmetic triangle groups.
\newblock {\em Journal of the Mathematical Society of Japan}, 29(1):91--106,
  1977.

\bibitem{Takeuchi77}
Kisao Takeuchi.
\newblock Commensurability classes of arithmetic triangle groups.
\newblock {\em J. Fac. Sci. Univ. Tokyo Sect. IA Math}, 24(1):201--212, 1977.

\bibitem{vanderPoortenWilliams}
Alfred van~der Poorten and Kenneth~S. Williams.
\newblock Values of the {D}edekind eta function at quadratic irrationalities.
\newblock {\em Canad. J. Math.}, 51(1):176--224, 1999.

\bibitem{Yang-Schwarzian}
Yifan Yang.
\newblock Schwarzian differential equations and {H}ecke eigenforms on {S}himura
  curves.
\newblock {\em Compositio Mathematica}, 149(1):1--31, 2013.

\bibitem{Zagier}
Don Zagier.
\newblock Traces of singular moduli.
\newblock In {\em Motives, polylogarithms and {H}odge theory, {P}art {I}
  ({I}rvine, {CA}, 1998)}, volume~3 of {\em Int. Press Lect. Ser.}, pages
  211--244. Int. Press, Somerville, MA, 2002.

\bibitem{Zudilin}
Wadim Zudilin.
\newblock Ramanujan-type formulae for {$1/\pi$}: a second wind?
\newblock In {\em Modular forms and string duality}, volume~54 of {\em Fields
  Inst. Commun.}, pages 179--188. Amer. Math. Soc., Providence, RI, 2008.

\end{thebibliography}
\end{document}